\documentclass[11pt]{article}
\usepackage{latexsym}
\usepackage{amssymb}
\usepackage{amsmath}
\usepackage{amscd}
\usepackage{array}
\usepackage{amsthm}
\usepackage{epsfig}
\usepackage{mathtools}
\usepackage{graphicx}

\newcommand{\letitre}{A theory of multiholomorphic maps}

\usepackage[margin=2cm,includeheadfoot]{geometry}

\usepackage{fancyhdr}
\fancypagestyle{plain}{%
  \fancyhead[LO]{}
  \fancyhead[CO]{}\fancyhead[RO]{}
}

\numberwithin{equation}{section}

\newcommand{\Hom}{\textrm{Hom}}
\renewcommand{\Im}{\mathrm{Im}}
\renewcommand{\ker}{\mathrm{ker}}
\renewcommand{\Re}{\mathrm{Re}}

\newcommand{\trace}{\mathrm{tr}}
\newcommand{\dVol}{\mathrm{dVol}}
\newcommand{\Vol}{\mathrm{Vol}}
\newcommand{\Gtwo}{\mathrm{G}_2}
\newcommand{\Speven}{\mathrm{Spin}_7}

\newcommand{\Sgn}{\mathrm{Sgn}}

\newcommand{\R}{\mathbb{R}}
\newcommand{\C}{\mathbb{C}}

\newcommand{\Oct}{\mathbb{O}}

\newcommand{\altemph}[1]{{\bf \emph{#1}}}

\newcommand{\End}{\mathrm{End}}

\newcommand{\citep}[2]{\cite[#2]{#1}}
\newtheorem{thm}{Theorem}[section]
\newtheorem{Defn}[thm]{Definition}
\newtheorem{Prop}[thm]{Proposition}
\newtheorem{Thm}[thm]{Theorem}
\newtheorem{Spec}[thm]{Speculation}
\newtheorem{Cor}[thm]{Corollary}
\newtheorem{Lem}[thm]{Lemma}
\newtheorem{Example}[thm]{Example}

\title{A theory of multiholomorphic maps}
\author{Aaron M. Smith}

\begin{document}
\maketitle
\abstract{This paper presents and explores a theory of \emph{multiholomorphic maps}.  This group of ideas generalizes the theory of pseudoholomorphic curves in a direction suggested by consideration of the kinds of compatible geometric structures that appear in the realm of special holonomy as well as some of the topological and analytic considerations that are essential to pseudoholomorphic invariants.  The first part presents the geometric framework of compatible $n$-triads, from which follows naturally the definition of a multiholomorphic mapping.  Some of the general analytic and differential-geometric properties of these maps are derived, including an energy identity which expresses a multiholomorphic map as a minimizer in its homotopy class of the appropriate $L^p$-energy.  Some theorems confining the critical loci of such maps are obtained as well as some Liouville-type theorems for maps with sufficient regularity in the presence of curvature hypotheses.  Finally, attention is focused onto a special case of the theory which pertains to the calibrated geometry of $\Gtwo$-manifolds.}  



\begin{section}{Introduction}
In recent decades the theory of pseudoholomorphic curves has been employed very fruitfully as a means of generating invariants of symplectic manifolds.  This approach, initiated by Gromov, begins by making a choice of a compatible trio of geometric structures on the target symplectic manifold.  Such a triple consists of a symplectic form $\omega$, a Riemannian metric $g$, and an almost complex structure $J$.  These elements are chosen to be compatible as a triad in the sense that any two determine the other uniquely.  The possibility of such compatibility is contained in the well-known ``2-out-of-3 property'' that relates $Sp(n,\R), SO(2n,\R)$, and $GL(n,\C)$.   Once such a choice is made, one can study solutions to the non-linear Cauchy--Riemann (CR) equation on maps between a Riemann surface and the target almost complex manifold.  Even thought the CR equation deals only with the complex structures, the geometric/topological interest of the theory is wound up in the way the symplectic structure ``tames'' the complex structure on the target.
  
This article begins by introducing a more general triad structure than what appears in the almost K\"{a}hler situation just described.  Such a triad includes three elements: a \emph{nondegenerate} $(n+1)$-form, an $n$-fold \emph{split product}, and a Riemannian metric along with some relevant compatibility stipulations.  One important consequence of these compatibility conditions is the existence of a \emph{vector cross product}.   

This setup naturally suggests a differential operator $\eth$, the \emph{multi-Cauchy--Riemann operator}, acting on maps between manifolds possessing compatible $n$-triads.  This operator measures the degree to which a map intertwines the vector cross products.  Solutions to $\eth u = 0$ are called \emph{multiholomorphic maps}, and coincide with pseudoholomorphic maps in the almost K\"ahler setting.  $\Gtwo$ and $\Speven$ manifolds also provide natural targets for multiholomorphic maps, and images of multiholomorphic maps realize associative and Cayley submanifolds in these settings respectively.

As in the pseudoholomorphic setting, a multiholomorphic map minimizes the appropriate $L^p$-energy within its homotopy class.  This variational perspective leads to regularity results on multiholomorphic maps and connects this project to the study of $p$-harmonic maps between Riemannian manifolds.

The main goal of this paper is to investigate this interesting framework first in general, and then more particularly in the setting in which a $3$-manifold is mapped into a torsion-free $\Gtwo$-manifold target.  This work looks forward to the realization of invariants generated by the topology of the moduli spaces of multiholomorphic maps, and/or applications to existence results for calibrated currents.

Here is a roadmap:
\begin{enumerate}
\item[$\bullet$] Section \ref{SubSecTriads} develops the definition of $n$-triads in general.

\item[$\bullet$]Section \ref{SubSecClass} gives a classification of such geometric structures on Riemannian manifolds by virtue of the fact that manifolds with vector cross product were classified by Brown and Gray \cite{BrownGrayVCP}.  

\item[$\bullet$]Section \ref{SubSecMCRE} defines the notion of a multiholomorphic map between manifolds equipped with compatible triads.  This notion is the primary novelty of the paper.  The relation to calibrated geometry is explained. 

\item[$\bullet$]Section \ref{SubSecCalGeom} demonstrates the connection with calibrated geometry, namely that multiholomorphic images are calibrated submanifolds (or currents).

\item[$\bullet$]Section \ref{SubSecEnergy} introduces a couple of important types of energy associated to maps between manifolds with compatible triads.  Most importantly, the notion of $L^{n+1}$-energy is introduced and it is shown that for any suitably differentiable map, this energy is equal to a positive term involving the multi-CR operator plus an integral of the pullback of the $(n+1)$-form from the target.  

\item[$\bullet$]Section \ref{SubSecVariation} investigates the variational problem associated to the $L^{n+1}$-energy.  It is seen that multiholomorphic maps are solutions to the $(n+1)$-Laplace equation ---the Euler--Lagrange equations associated to the energy functional--- and in particular are minimizers of the $L^p$-energy.

\item[$\bullet$]Section \ref{SubSecEndo} discusses the relation between multiholomorphic endomorphisms and quasiregular maps.  

\item[$\bullet$]Section \ref{SubSecCrit} considers the critical locus of a multiholomorphic map.  It is shown that a non-constant multiholomorphic endomorphism is a local homeomorphism, and hence a conformal cover.  Secondly, it is shown that the critical locus of an arbitrary multiholomorphic map has Hausdorff codimension at least $2$. 


\item[$\bullet$]Section \ref{SubSecMainGoals} describes the general motivation for the the study of multiholomorphic maps.

\item[$\bullet$]Section \ref{SubSecGtwo} delves into the case involving multiholomorphic maps from a $3$-manifold into a $\Gtwo$-manifold.  This is one of the first cases in which the multiholomorphic framework is completely novel.  After recapitulating important properties of $\Gtwo$-manifolds, there is some discussion of existence of solutions and the overdeterminedness of the MCR equations.  What remains of the paper is a proof of a Liouville-type theorem which constrains the critical locus of a multiholomorphic map from a $3$-manifold into a $\Gtwo$-target.  
\end{enumerate}
\end{section}

\begin{section}{$n$-triads} \label{SubSecTriads}  
We start this article by introducing a slightly more general triadic structure than what appears in the almost K\"{a}hler situation described in the introduction.  Such a triad includes three elements: a \emph{nondegenerate} $(n+1)$-form, an $n$-fold \emph{split product}, and a Riemannian metric along with some relevant compatibility stipulations.  The split product is \emph{not} a stand-alone generalization of an almost complex structure: it is not the same notion as an almost complex structure when $n=1$.  However, after imposing the list of compatibility conditions, the notions coincide when $n=1$.  In general the compatibility conditions constrain the split product to being a \emph{vector cross product }(VCP).  Such objects were classified by Brown and Gray in \cite{BrownGrayVCP}, and have been an important object of study in some recent work of Leung et. al; see \cite{LeeLeung08, LeeLeung09, Leung02}.  The latter sources have some close connections to this endeavor.  They are concerned with presenting a general notion of \emph{instantons} in manifolds with vector cross products --that is, submanifolds of a manifold with vector cross product which would arise as particularly nice images of the multiholomorphic maps defined here.  


\begin{Defn}
An \altemph{$n$-fold split product} on a smooth manifold $M$ is a pair $(J,K)$ of bundle homomorphisms (over the identity)
\begin{equation*}
 K\colon TM \rightarrow \Lambda^nTM, \,\,\,\,\, J\colon \Lambda^nTM \rightarrow TM
\end{equation*}
such that
\begin{equation*}
 J \circ K = (-1)^n \lambda I_{TM}
\end{equation*}
for some positive constant factor $\lambda >0$.  This condition is the same as saying the exact sequence associated to $J$, 
\begin{equation*}
 0 \rightarrow \ker(J) \rightarrow \Lambda^nTM \rightarrow TM \rightarrow 0 
\end{equation*}
is split by $\frac{(-1)^n}{\lambda} K$.
\end{Defn}

With a metric in hand, one has an induced metric $\Lambda^n g$ on $\Lambda^nTM$ by application of $\Lambda^n$ functorially:
\begin{align*}
&\Lambda^n \colon \Hom_{\R}(V,W) \rightarrow \Hom_{\R}(\Lambda^nV, \Lambda^n W) \\
&\Lambda^n\phi(v_1 \wedge \dotsc \wedge v_n) = n\sum_{\sigma \in S_n} \frac{1}{n!} \Sgn(\sigma) \phi(v_{\sigma(1)})\wedge \dotsc \wedge \phi(v_{\sigma(n)})) = n \phi(v_1) \wedge \dotsc \wedge \phi(v_n).
\end{align*}

\begin{Defn}
A bundle homomorphism $J$ on a Riemannian manifold $(M,g)$
\begin{equation*}
 J \colon \Lambda^nTM \rightarrow TM
\end{equation*}
is called a \altemph{vector cross product} if it generates with $g$ an $(n+1)$-form $\omega$,
\begin{equation*}
 g(A_0,J(A_1, \dotsc, A_n)) = \omega(A_0, \dotsc, A_n), \,\,\,\,\, \omega \in \Omega^{n+1}TM
\end{equation*}
and has comass $= 1$ in the sense
\begin{equation*}
 \lVert J(A_1, \dotsc, A_n) \rVert_g^2 = \lVert A_1 \wedge \dotsc \wedge A_n \rVert^2_{\Lambda^n g}.
\end{equation*}
\end{Defn}

There is a close connection between vector cross products and calibrated geometry because a vector cross product on a Riemannian manifold yields a calibrating form via $\omega = g(J\bullet, \bullet)$ under the additional stipulation that $d\omega = 0$.

Generalizing the notion of non-degeneracy for $2$-forms we have, 
\begin{Defn}
An $(n+1)$-form $\omega$ is called \altemph{non-degenerate} if for all non-zero tangent vectors $V$, at any tangent space, the contraction map
\begin{equation*}
\iota_V \colon T_xM \rightarrow \Lambda^nT^*_xM  
\end{equation*}
is injective. 
\end{Defn}
The condition of nondegeneracy is an open condition.  

We have all the ingredients to make the definition of a compatible triad.
\begin{Defn}
 An \altemph{$n$-compatible triad} on $M$ is a triple $(\omega,g,(J,K))$ consisting of a non-degenerate $(n+1)$-form $\omega$, a Riemannian metric $g$, and an $n$-split product $(J,K)$ such that

\begin{enumerate}
\item[(i)] $J$ is a vector cross product with respect to $g$,
\item[(ii)]
 $\omega(\zeta,B) = g(J(\zeta),B)$, and
\item[(iii)]
 $\Lambda^ng(\zeta, KA) = \omega(A,\zeta).$
\end{enumerate}
\end{Defn}

The conditions (ii), (iii) imply that $J,K$ are adjoint to each other up to $(-1)^n$ with respect to the metrics $g$, and $\Lambda^n g$.  Along with $J\circ K = (-1)^n \lambda I$ these conditions imply
\begin{equation*}
 g(A_0,A_1)  = \lambda^{-1} (-1)^n\omega(K(A_0),A_1).
\end{equation*}

\begin{Defn}
 We say a compatible triad is \altemph{$n$-plectic} if $\omega$ is closed and non-degenerate.  We refer to a manifold with such a form as \altemph{$n$-plectic}, and a triad containing such a form an \altemph{$n$-plectic triad}.
\end{Defn}
The definition of an $n$-plectic or multi-(sym)plectic form appears in work of Gotay, \cite{Gotay91}, and a recent paper by Baez, Hoffnung, and Rogers  \cite{MR2566161}.  In both cases these notions are motivated by the canonical $n$-plectic form on an $n$-form bundle.  Symplectic manifolds are $1$-plectic.

Because a triadic manifold is Riemannian it makes sense to consider the covariant derivatives of any of these tensors.  Hence,
\begin{Defn}
 Let $\nabla$ be the Levi-Civita connection associated to $g$.  Then, we say $\omega$, and $(J,K)$ are \altemph{parallel} if \begin{equation*}                                                    
\nabla \omega = 0.
\end{equation*}
\end{Defn}
An direct calculation shows that $\nabla \omega = 0$ implies $\nabla J = \nabla K = 0$.

In the examples we are about to consider, all the relevant tensors are parallel, and the $(n+1)$-forms are, in particular, closed.  But these conditions are all integrability issues and are not necessary according to the definitions above.
\end{section}
\begin{section}{Classification}\label{SubSecClass}
Brown and Gray's classification of vector cross products (VCP) on $\R^n$ yields a classification of compatible triads because the data of a VCP determines a compatible triad uniquely.  Hence the following list of examples is in fact \emph{exhaustive} modulo integrability considerations. That is to say that since their classification is a project of linear algebra on $\R^n$, it extends to a classification of these kinds of tensor structures on a smooth manifold but naturally cannot say anything about differential properties.

\begin{Example}(Hermitian Triad)
 Suppose $M$ is (almost) Hermitian.  The nondegenerate form $\omega$, (almost) complex structure $J$, and Riemannian metric $g$ form the paradigmatic example of a compatible triad ---in this case a $1$-triad.  If $\omega$ has the integrability condition $d\omega = 0$ then the triad is called an (almost) K\"{a}hler triad.
\end{Example}

\begin{Example}(Conformal Triad) \label{ConfTriads}
Suppose $M$ is an orientable, $(n+1)$-dimensional smooth, Riemannian manifold, and consider the case when the $n$-plectic form is a volume form.  Then the metric $g$, and volume form $\dVol_M$ (which is $n$-plectic) yields a canonical $n$-split product $(J,K)$  which fits into a triad.  Namely, let $J$ and $K$ be particular multiples of the Hodge dual operator on $TM$.  Explicitly,
\begin{equation*}
 J = (-1)^n \star, \,\,\,\,\, K = (-1)^n \star.
\end{equation*}
The properties of the Hodge star imply
\begin{equation*}
K \circ J = (-1)^n I, \,\,\,\,\, J \circ K = (-1)^n I,
\end{equation*}
and consequently that the Hodge star operation acting on forms on $M$ is (in degrees $1,n$) precomposition with $(-1)^nJ,(-1)^nK$.  
\end{Example}
This example is crucial because such a triadic manifold is often the domain of the maps we consider in subsequent sections.

\begin{Example}(Associative Triad)(See \cite{HarveyLawson82} for details.)\\
Consider the octonions $\Oct$, and fix an identification with $\R^8$:
\begin{equation*} 
 x^0 + x^1I + x^2J + x^3K + y^01^\prime + y^1I^\prime + y^2J^\prime + y^3K^\prime = (x^0,x^1,x^2,x^3,y^0,y^1,y^2,y^3).
\end{equation*}
$\Oct$ is equipped with the octionic product $\cdot$, the notion of real/imaginary parts, conjugation, and the usual Euclidean inner product $g(,)$: 
\begin{align*}
 & \Re(a + bI + cJ + dK + e1^\prime + fI^\prime + gJ^\prime + hK^\prime) = a, \\
 & \Im(a + bI + cJ + dK + e1^\prime + fI^\prime + gJ^\prime + hK^\prime) = bI + cJ + dK + e1^\prime + fI^\prime + gJ^\prime + hK^\prime,\\
 & \bar{A} := \Re A - \Im A,\\
 &g(A,B) := \Re(A \cdot \bar{B}).
\end{align*}
Furthermore, the octonionic product restricts to a product on $\Im\Oct = \R^7$ after projection:
\begin{equation*}
 J(A,B) := \Im(A \cdot B).
\end{equation*}
From these structures we can define an alternating three-form on $\R^7$ by
\begin{equation*}
 \omega(A,B,C) = g(J(A,B),C).
\end{equation*}
In coordinates, we have the standard volume form and metric, and
\begin{equation*}
 \omega_0 = dx^{123} - dx^1(dy^{23} + dy^{10})-dx^2(dy^{31} + dy^{20})-dx^3(dy^{12} + dy^{30}).
\end{equation*}
It is worth noting that this three-form furnishes a well-known description of the exceptional Lie group $\Gtwo$:
\begin{equation*}
 \Gtwo := \{ \sigma \in GL(\Im\Oct) \mid \sigma^* \omega_0 = \omega_0 \}
\end{equation*}
In \cite{Bryant87}, Bryant proved that the exceptional Lie group $\Gtwo$ can also be described as the group which preserves the metric and vector cross product.  For this reason, a Riemannian $7$-manifold with $\Gtwo$-holonomy is equipped with a parallel $3$-form defined via parallel transport, and a vector cross product given by
\begin{equation*}
 g(J(A,B),C) = \omega(A,B,C).
\end{equation*}
The other half of the split product, $K$ is the $g$-adjoint of $J$ and satisfies
\begin{equation*}
 g(A,B) = \frac{1}{3} \omega(K(A),B).
\end{equation*}
So any $\Gtwo$-manifold has a canonical $2$-triad.  
\begin{Defn}
Such a triad on a $\Gtwo$-manifold is called the \altemph{associative triad}.  If $\nabla \omega = 0$, $M$ is said to have a \altemph{torsion-free $\Gtwo$-structure}, a feature which is equivalent to the condition that $\omega$ is harmonic.  If $\omega$ is closed but not necessarily co-closed, the structure is called a \altemph{closed $\Gtwo$-structure}, and the terminology \altemph{$\Gtwo$-structure} denotes merely the choice on $M$ of a principal $\Gtwo$-subbundle of the frame bundle on $M$.  The latter condition implies the existence of a triad which is not necessarily closed.  See section \ref{SubSecGtwo} for a more full discussion of these facts.
\end{Defn}
\end{Example}

\begin{Example}(Cayley Triad)(See \cite{HarveyLawson82} for details.)\\
A $Spin_7$-manifold is precisely one whose tangent spaces are coherently identified with $\Oct$.  This identification means that the manifold is equipped with a triple product coming from the triple product on $\Oct$,
\begin{equation*}
 x \times y \times z := \frac{1}{2}(x(\bar{y}z)-z(\bar{y}x)).
\end{equation*}
This product defines a 4-form called the Cayley calibration:
\begin{equation*}
 \Phi(x\wedge y\wedge z \wedge w) := g(x \times y \times z,w) = g(J(x,y,z),w),\,\,\,\,\, J(A,B,C) := A \times B \times C.
\end{equation*}
As in the $\Gtwo$-case, these fit into a $3$-triad we call the \altemph{Cayley triad}.
\end{Example}
\end{section}

\begin{section}{The Multi-Cauchy--Riemann Equation} \label{SubSecMCRE}
The primary novelty of this paper is the definition and study of a non-linear PDE we call the \emph{multi-Cauchy--Riemann Equation}.  The framework of compatible triads developed in the previous sections naturally suggests such an equation as a direct generalization of the Cauchy--Riemann equation on maps from Riemann surfaces into almost-K\"ahler manifolds.

Suppose a smooth manifold $M$ is equipped with an $n$-plectic triad $(M,\omega,g,(J,K))$.  And let $X$ be a closed, compact $(n+1)$-manifold with an oriented Riemannian structure ---hence a conformal $n$-triad $(X,\dVol_X,g_X,(j=(-1)^n\star,k=(-1)^n\star))$.  

Let $u$ be a differentiable map, $u\colon X \rightarrow M$.  We use the notation $\Lambda^n du$ to denote the application of the $n$-th exterior power functor pointwise to $du$.  $\Lambda^n$ is defined by,
\begin{align*}
&\Lambda^n \colon \Hom_{\R}(V,W) \rightarrow \Hom_{\R}(\Lambda^nV, \Lambda^n W) \\
&\Lambda^n\phi(v_1 \wedge \dotsc \wedge v_n) = n\sum_{\sigma \in S_n} \frac{1}{n!} \Sgn(\sigma) \phi(v_{\sigma(1)})\wedge \dotsc \wedge \phi(v_{\sigma(n)})) = n \phi(v_1) \wedge \dotsc \wedge \phi(v_n).
\end{align*}
Consequently $\Lambda^n du$ is an element of $\Omega^n(X,\Lambda^n u^*TM)$.  

Let the function $\lvert du \rvert$ be the norm of $du$ as a vector pointwise in the metric tensor product space $T_x^*X \otimes u^*T_{u(x)}M$, that is, the Hilbert--Schmidt norm pointwise.  Explicitly, the metric is $g \otimes g_X^*$, so that 
\begin{equation*}
\lvert du \rvert = (g\otimes g_X^*(du,du))^{\frac{1}{2}}.                                                                                                              
\end{equation*}

Now we introduce a dynamical equation that mimics the Cauchy--Riemann equation:
\begin{Defn}
Let $u$ be a map $u \colon X \rightarrow M$.  The \altemph{multi-Cauchy--Riemann operator} acting on $u$ is
\begin{equation}
 \eth u := \frac{1}{(n+1)^{\frac{n-1}{2}}}\lvert du \rvert^{n-1}  du -\frac{(-1)^n}{n}J \circ \Lambda^n du \circ k.
\end{equation}
The \altemph{Multi-Cauchy--Riemann Equation (MCRE)} is
\begin{equation}
 \eth u = 0.
\end{equation}
A solution is called a \altemph{multiholomorphic map}.
\end{Defn}

It is worth noting that there is a pre-existing notion due to Gray \cite{GrayVCPPreserving} of a map which is vector cross product-preserving.  In the established notation, this would take the form
\begin{equation*}
 du \circ j = \frac1n J \circ \Lambda^n du,
\end{equation*}
which ultimately forces solutions to be isometric immersions.  This is a stronger condition that does not allow the conformal invariance that the multiholomorphic equation evidently does.  In this vein, we have the lemma,

\begin{Lem}\label{AutomLem}
 A endomorphism $u \colon X \rightarrow X$ on an $n+1$-dimensional manifold with conformal $n$-triad (Example \ref{ConfTriads}) is an orientation-preserving conformal map if and only if it is a multiholomorphic local homeomorphism.
\end{Lem}
\begin{proof}
 Let $u$ be a conformal endomorphism of $X$.  Since $u^*g = \frac{1}{n+1}\lvert du \rvert^2 g$, the conformal factor is the square of
 \begin{equation*}
 \lambda := \frac{1}{(n+1)^{\frac{1}{2}}}\lvert du \rvert. 
 \end{equation*}
Let $\{ e_i \}$ be a positively-oriented, orthonormal basis for the tangent space at any point of $X$.  By the definition of the Hodge star, 
\begin{equation*}
\star(e_i) = (-1)^i e_0 \wedge \dotsc \wedge e_{i-1} \wedge e_{i+1} \wedge \dotsc \wedge e_{n}.                                                                                                                             
\end{equation*}
By hypothesis, $\{ \frac{1}{\lambda} du(e_i) \}$ is a positively-oriented, orthonormal frame for the tangent space at $u(x)$.  In the target we write 
\begin{equation*}
\hat{du(e_i)} := \sum du(e_0) \wedge \dotsc \wedge du(e_{i-1}) \wedge du(e_{i+1}) \wedge \dotsc \wedge du(e_n).                                                                                                                                     
\end{equation*}
Then, it follows similarly that $\star(\frac{1}{\lambda} \hat{du(e_i)}) = \frac{(-1)^{n-i}}{\lambda}du(e_i)$.  Putting these observations together we see
\begin{equation*}
 \frac{1}{\lambda^n} \star(\frac{1}{n}\Lambda^n du \star(e_i)) = \frac{(-1)^n}{\lambda} du(e_i).
\end{equation*}
Which is the same as,
\begin{equation*}
 \frac{\lvert du \rvert^{n-1}}{(n+1)^{\frac{n-1}{2}}}du(e_i) - \frac{(-1)^n}{n} j \circ \Lambda^n du \circ k(e_i) = 0.
\end{equation*}
To prove the other direction, we take the formula just above as the hypothesis and consider the term $\frac{1}{\lambda^2}g(du(e_i),du(e_i))$.  Since $\star$ is an isometry both for $g$ and $\Lambda^n g$, this term is merely, $g(e_i,e_i) = 1$.  Plugging in different vectors yields $g(e_i,e_j) = \delta_{ij}$.  The The following more general lemma makes the same argument in more detail. 
\end{proof}
The last observation in the lemma can be extended to the more general claim
\begin{Lem}\label{Ortholem}
 Let $u \colon X^{n+1} \rightarrow M$ be an arbitrary multiholomorphic map.  In the notation of the previous lemma, at a regular point $x$ of $u$, the vectors $\{ \frac{1}{\lambda}du(e_i)\}$ are an oriented, orthonormal $(n+1)$-frame in $T_{u(x)}M$.
\end{Lem}
\begin{proof}
At $x$ the MCRE is, for any $i$,
 \begin{equation*}
 \lambda^{n-1}du(e_i) = \frac{(-1)^n}{n} J \circ \Lambda^n du \circ k(e_i).
\end{equation*}
Hence,
\begin{equation*}
 \frac{1}{\lambda^{2}}g(du(e_i),du(e_i)) = \frac{\lambda^{n+1}}g(J\hat{du(e_i)},du(e_i)) = \frac{1}{\lambda^{n+1}}\omega(du(e_0) \wedge \dotsc \wedge du(e_n)).
\end{equation*}
The RHS doesn't depend on $i$, and by the anti-symmetry of $\omega$, $g(du(e_i),du(e_j))=0$ if $i \neq j$.  And, since all of these vectors are the same length, then $\lambda = \lvert du(e_i)\rvert_g$ for any $i$.
\end{proof}


\begin{Prop}
 The MCR equation is conformally equivariant.  More precisely, suppose 
\begin{equation*}
\phi\colon X^{n+1} \rightarrow X^{n+1}
\end{equation*}
is an orientation-preserving conformal homeomorphism with conformal factor $\mu^2 > 0$, that is $\phi^* g = \mu^2 g$.  Then
\begin{equation*}
 \eth (u \circ \phi) = \mu^{(n-1)} (\eth u) \circ d\phi.
\end{equation*}
As a result the solution space is conformally invariant.
\end{Prop}
\begin{proof}
First note how the Hodge star transforms:
\begin{equation*}
 \Lambda^n d\phi \circ k = \mu^{(n-1)}k \circ d\phi.
\end{equation*}
Then we have,
\begin{align*}
  \eth (u\circ \phi) &= \frac{1}{(n+1)^{\frac{n-1}{2}}}\lvert du \circ d\phi \rvert^{n-1}  du \circ d\phi -\frac{(-1)^n}{n}J \circ \Lambda^n (du \circ d\phi) \circ k \\
 &= \frac{1}{(n+1)^{\frac{n-1}{2}}} \mu^{(n-1)}\lvert du \rvert^{n-1} du \circ d\phi -\frac{(-1)^n}{n}J \circ \Lambda^n du \circ \frac{1}{n}\Lambda^n d\phi \circ k
\end{align*}
Lemma \ref{AutomLem} implies $\frac{1}{n}\Lambda^n d\phi \circ k = \mu^{n-1}k \circ d\phi$.  Plugging this into the right-hand term above implies the result, 
\begin{align*}
 &= \frac{1}{(n+1)^{\frac{n-1}{2}}} \mu^{(n-1)}\lvert du \rvert^{n-1} du \circ d\phi -\frac{\mu^{(n-1)}(-1)^n}{n}J \circ \Lambda^n du \circ k \circ d\phi.
\end{align*}
\end{proof}

In the above calculation we used the observation that multiholomorphic automorphisms $\phi \colon X \rightarrow X$ are merely orientation-preserving conformal automorphisms.  This coincidence demonstrates an important coherence of the framework of $n$-triads: precomposing a multiholomorphic map from $X$ to $M$, with an automorphism of $X$ preserves the solution set to the MCR equation between $X$ and $M$.  Therefore if $X,M$ are two $n$-triadic $(n+1)$-dimensional manifolds, then the multiholomorphic maps between them includes the set of orientation-preserving conformal maps between $X$ and $M$.  But a general solution can be singular (and hence are \emph{orientation non-reversing}).  A useful way to think of such maps is as conformal maps except that the conformal factor can vanish on some branch locus.  In the local, flat setting this notion exactly coincides with the well-studied field of $1$-quasiregular mappings.  In general, multiholomorphic maps between triadic manifolds with conformal triads can be regarded as the appropriate geometric generalizations of the notion of $1$-quasiregular mappings.  This issue is discussed more closely in section \ref{SubSecEndo}, and the reader is further encouraged to consult \cite{IwaniecMartin01} ---an excellent presentation of the modern issues in this direction.  An alternate notion is the concept in Riemannian geometry of a weakly conformal map. 
\begin{Defn}
 A map between Riemannian manifolds $u \colon (X,g_X) \rightarrow (M,g)$ is called weakly conformal if $u^*g = \lambda g_X$ for some non-negative function $\lambda$. 
\end{Defn}
Let $u \colon X \rightarrow M$ be a multiholomorphic map with the usual notation.  One can easily compute that
\begin{equation*}\label{weakconformal}
 u^*g = \frac{1}{n+1}\lvert du \rvert^2g_X.
\end{equation*}
Hence, any multiholomorphic map is, in particular, weakly conformal with weak conformal factor $\lambda :=\frac{ \lvert du \rvert^{n-1}}{n+1}.$  This observation is crucial to the unique continuation argument of section \ref{SubSecCrit}.

In what remains in this section we describe the MCR equations in the presence of isothermal coordinates.  This vantage lends some perspective about what is going on locally ---at least in the locally conformally flat case. 

Suppose $X$ is locally conformally flat (LCF).  We have local conformal coordinates given by the chart $\phi^\alpha\colon U \subset X \rightarrow \R^{n+1}$, with $\psi^\alpha := (\phi^\alpha)^{-1}$.  Then, a map $u$ satisfies the MCR equation if and only if in local conformal coordinates, $u^\alpha := u \circ \psi^\alpha$ satisfies the MCR system
\begin{equation}
 \frac{1}{(n+1)^{\frac{n-1}{2}}}(\sum^n_{j=0} \lvert \partial_j u \rvert_g^2)^{\frac{n-1}{2}} \partial_iu - (-1)^{n(n-i)}J( \partial_{i+1}u \wedge \dotsc \wedge \partial_{i-1}u) = 0.
\end{equation}
The indices in the right-most term are understood to increase in order from $i+1$ to $i-1$ modulo $n+1$.  Applying $g(\partial_i u,\bullet)$ to both sides yields
\begin{equation*}
 C\left\lvert \partial_i u \right\rvert^2 = (-1)^{n(n-i)}g(\partial_i u,J(\partial_{i+1}u \wedge \dotsc \wedge \partial_{i-1}u))
 = \omega(\partial_0 u\wedge \dotsc \wedge \partial_n u) 
\end{equation*}
for some $C$ subsuming all of the coefficients.  So $\lvert \partial_i u \rvert = \lvert \partial_ju \rvert$ for all $i$,$j$.  Similarly, by the same kind of argument all of these partials $\partial_i u $ must be mutually orthogonal because of the fact that $\omega() = g(,J)$ is an alternating form.
Hence, we can simplify slightly in these conformal coordinates:
\begin{Prop}
 In local conformal coordinates $\{x_0, \dotsc, x_n\}$, the MCR equation on $u\colon U \subset \R^n \rightarrow M$ is satisfied if and only if, $\{ \partial_i u \}$ is a positively-oriented orthogonal frame in $T_{u(x)}M$ for all $x$, and $\lvert \partial_i u \rvert = \lvert \partial_j u \rvert$ for all $i,j$.  Hence, we could write the equations (with indices understood mod $n+1$)
\begin{equation*}
 \lvert \partial_i u \rvert^{\frac{n-1}{2}}\partial_i u = (-1)^{n(n-i)}J(\partial_{i+1}u \wedge \dotsc \wedge \partial_{i-1}u).
\end{equation*}
\end{Prop}

\end{section}  

\begin{section}{Interaction with Calibrated Geometry} \label{SubSecCalGeom} 
We recall the definition which founds the subject of calibrated geometry.  Let $M$ be a compact, closed, Riemannian manifold.  
\begin{Defn}
 Let $\omega$ be a differential $k$-form on $M$.  If $\omega$ is closed, and satisfies the metric condition that at any point in $M$, for any positively-oriented $k$-plane, $\zeta_1 \wedge \dotsc \wedge \zeta_k$,
\begin{equation*}
 \omega(\zeta_1 \wedge \dotsc \wedge \zeta_k) \leq dVol(\zeta_1 \wedge \dotsc \wedge \zeta_k),
\end{equation*}
then $\omega$ is called a \altemph{calibration}, and $M$ equipped with $\omega$ a \altemph{calibrated manifold}.
\end{Defn}

\begin{Lem}
If $M$ is equipped with a compatible $n$-triad, then it is a calibrated manifold with an $(n+1)$-calibration 
\end{Lem}
\begin{proof}
This follows from the well-known fact that the existence of a vector cross product implies that $\omega$ is a calibration, for instance see \cite{LeeLeung08}. 
\end{proof}

A calibration determines a distinguished class of submanifolds.
\begin{Defn}
 Let $M,g,\omega$ be a Riemannian manifold with $k$-calibration $\omega$.  A $k$-dimensional submanifold on which $\omega$ restricts to the $k$-dimensional volume form is called a \altemph{calibrated submanifold}.
\end{Defn}
One might consider the distribution in $\Lambda^kTM$ consisting of oriented $k$-planes on which $\omega$ restricts to the volume form.  Then the calibrated submanifolds are the submanifolds which are integral for this distribution. 

The fundamental connection between multiholomorphic maps and calibrated geometry is the following lemma.  Keeping the usual notation, 
\begin{Lem}
 Let $u \colon X^{n+1} \rightarrow M^{m}$ be a multiholomorphic embedding.  Then, $\mathrm{image}(u)$ is a calibrated submanifold in $M$.  If $u$ is an arbitrary multiholomorphic map, then $\mathrm{image}(u)$ is a calibrated current.
\end{Lem}
\begin{proof}
 The first claim follows by observing the MCRE equations at the origin of local normal coordinates.  Let $\{ e_i \}$ be the oriented orthonormal coordinate frame at the origin.  We have, by Lemma \ref{Ortholem}, that the set $\{du(e_i)\}$ is orthogonal.  And by the calculations in the proof of that Lemma,
\begin{align*}
 u^*\omega(e_0 \wedge \dotsc \wedge e_n) =& \omega(du(e_0) \wedge \dotsc \wedge du(e_n)) = \lambda^{n-1}g(du(e_i),du(e_i))= \lambda^{n+1}.
\end{align*}
Hence, 
\begin{equation*}
 \omega(du(e_0) \wedge \dotsc \wedge du(e_n)) = \prod_{i} g(du(e_i),du(e_i)) = \mathrm{det}(g(du(e_i),du(e_j)))^{\frac12} = dVol_g(du(e_0)\wedge \dotsc \wedge du(e_n))
\end{equation*}
We have proven that $\omega$ restricts to the volume form on the tangent plane to the embedding at any point.

The second claim follows from the regularity theory for the $p$-Laplace operator ---solutions have $C^{1,\alpha}$-interior regularity (see section \ref{SubSecVariation}), and hence their images are currents.
\end{proof}

\end{section}

\begin{section}{Energy Identities}\label{SubSecEnergy}
The MCR equation is closely intertwined with a variational problem related to the $(n+1)$-energy:
\begin{Defn}

Given a compact smooth $(n+1)$-manifold $X$ and any smooth manifold $M$ both with compatible $n$-triads $(X,g_X,\dVol_X,(j,k))$, $(M,g_M,\omega,(J,K))$, the \altemph{$(n+1)$-energy} of a map $u\colon X \rightarrow M$ is
\begin{equation*}
 E_{n+1}(u) := \frac{1}{(n+1)^{\frac{n+1}{2}}}\int_{X} \lvert du \rvert^{n+1} \dVol_X. 
\end{equation*}
The integrand (including the coefficient) is called the \altemph{energy density} and is denoted $e_{n+1}$.
\end{Defn}
Note that if $X$ is not compact the energy could be infinite.
 
A first observation about multiholomorphic maps:
\begin{Prop} 
The energy of a multiholomorphic map whose domain is a compact, $n$-triadic, $(n+1)$-manifold is a topological invariant.  Specifically, let $u$ be a differentiable map
\begin{equation*}
 u\colon (X,g_X,\dVol_X,(j,k)) \rightarrow (M,g,\omega,(J,K))  
\end{equation*}
between a compact $n$-triadic $(n+1)$-manifold $X$ and an $n$-triadic manifold $M$ that satisfies $\eth u = 0$.  Then
\begin{equation*}
 E_{n+1}(u) = \int_X u^*\omega
\end{equation*}
\end{Prop}
\begin{proof}
Consider 
\begin{align*}
 \frac{1}{(n+1)^{\frac{n+1}{2}}}\int_X \lvert du \rvert^{n+1}\dVol_X &= \frac{1}{(n+1)^{\frac{n+1}{2}}} \int_X \lvert du \rvert^{n-1}(g \otimes g^*_X)(du,du)\dVol_X \\
 &= \int_X \frac{1}{(n+1)}(g \otimes g^*_X)(du,\frac{(-1)^n}{n}J \circ \Lambda^n du \circ k)\dVol_X   
\end{align*}
Since $\circ k = (-1)^n \star$, $g^*_X(\alpha,\beta)\dVol_X = \alpha \wedge (-1)^n\beta \circ k$ for any $\alpha$, $\beta$, $1$-forms on $X$.  So that
\begin{equation*}
 g \otimes g_X^*(du, \frac{(-1)^n}{n}J \circ \Lambda^n du \circ k)\dVol_X = g(du,\frac{(-1)^n}{n} J \circ \Lambda^n du) = (n+1)\cdot u^*\omega
\end{equation*}
Hence,
\begin{equation*}
 E_{n+1}(u)= \frac{1}{(n+1)^{\frac{n+1}{2}}}\int_X \lvert du \rvert^{n+1}\dVol_X = \int_X u^*\omega,   
\end{equation*}
as desired.
\end{proof}
This estimate shows that if one fixes the topological class of a family of multiholomorphic maps then these have a fixed, finite energy.  It also corresponds to a uniform $W^{1,n+1}$-bound on such a class.  However, it is not immediately obvious that this is a variational identity.  In the study of pseudoholomorphic curves we have the more robust equation for \emph{any} map $u$
\begin{equation}\label{PdoEI}
 \frac{1}{2}\int_\Sigma \lvert du \rvert^2 \dVol_\Sigma = \frac{1}{2}\int_X \lvert \bar{\partial}u \rvert^2 + \int_X u^*\omega
\end{equation}
in which both quantities on the right are positive.  This identity shows that a pseudoholomorphic curve \emph{minimizes} energy, and hence can be realized as a critical point of the energy functional.  This perspective yields (by virtue of the Euler--Lagrange equations associated to this energy functional) the important fact that a pseudoholomorphic curves are harmonic.

We now consider an alternate notion of energy we call the \emph{mixed energy} which leads to a direct generalization of the energy identity for pseudoholomorphic curves (\ref{PdoEI}) and immediately expresses a minimizing property.  In fact the $(n+1)$-energy is minimized by multiholomorphic maps although it was not obvious from the start; this fact follows from the calculations below.  The important result is summarized in the \emph{$(n+1)$-energy identity} at the end of the section.

\begin{Defn}
The \altemph{mixed energy} of a map $u \colon X \rightarrow M$ is
\begin{equation*}
 E_{mix}(u):= \int_X \left[\frac{1}{2(n+1)^{\frac{n+1}{2}}}\lvert du \rvert^{n+1} +  \frac{(n+1)^{\frac{n-3}{2}}}{2\lvert du \rvert^{n-1}}\left\lvert \frac{1}{n} \Lambda^n du \right\rvert^2\right]\dVol_X.
\end{equation*}
The integrand is called the \altemph{mixed energy density} and is denoted $e_{mix}(u)$.
\end{Defn}

Note that the term in $E_{mix}$ involving $\frac{\left \lvert \frac{1}{n}\Lambda^n du  \right \rvert}{\lvert du \rvert^{n-1}}$ does not blow up at critical points of $u$ (as long as $du$ is continuous) by virtue of the relevant Hadamard inequality \citep{IwaniecMartin01}{sec. 9.9}:
\begin{equation} \label{Hadamard}
 (n+1)^{\frac{n-1}{2}}\lvert \frac{1}{n}\Lambda^n du \rvert \leq \lvert du \rvert^n.
\end{equation}

We have the following important identity:
\begin{Thm}(Mixed Energy Identity)
For any differentiable map $u \colon X \rightarrow M$,
\begin{equation*}\label{MixedEnergyIdentity}
 E_{mix}(u) = \frac{(n+1)^{\frac{n-3}{2}}}{2}\int_X \frac{\lvert \eth u \rvert^2}{\lvert du \rvert^{n-1}}\dVol_X + \int_X u^*\omega.
\end{equation*}
\end{Thm}
\begin{proof}
We first calculate
\begin{equation} \label{equation1}
 \left\lvert \frac{1}{n} J \circ \Lambda^n du \circ k \right\rvert = \left\lvert \frac{1}{n} \Lambda^n du \right\rvert.
\end{equation}
Consider for any unit vector $\zeta_0$, a completion to an orthonormal basis $\{ \zeta_i \}$ for $T_xX$.  The Hodge star $(-1)^nk$ sends any $\zeta_j$ to $(-1)^{n+j} \zeta_0 \wedge \dotsc \wedge \hat{\zeta_j} \wedge \dotsc \wedge \zeta_n =: (-1)^n\hat{\zeta_j}.$  Then,
\begin{align*}
 \left\lvert \frac{1}{n}J \circ \Lambda^n du \circ k \right\rvert^2 = & \sum_i \left\lvert \frac{1}{n}J \circ \Lambda^n du \circ k(\zeta_i)\right\rvert^2\\
  =& \sum_i \left\lvert \frac{1}{n} J \circ \Lambda^n du (\hat{\zeta_i}) \right\rvert^2 = \sum_i \left\lvert \frac{1}{n}\Lambda^n du (\hat{\zeta_i})\right\rvert^2
  = \left\lvert \Lambda^n du \right\rvert^2.  
\end{align*}

Next, consider the pointwise calculation of the norm of $\frac{\eth u}{\lvert du \rvert^{\frac{n-1}{2}}}$.  
\begin{equation*}
(n+1)^{\frac{n-1}{2}} \frac{\left\lvert \eth u \right\rvert^2}{\lvert du \rvert^{n-1}} = \frac{1}{(n+1)^{\frac{n-1}{2}}}\lvert du \rvert^{n+1} +  \frac{(n+1)^{\frac{n-1}{2}}}{\lvert du \rvert^{n-1}}\left\lvert \frac{1}{n} J \circ \Lambda^n du \circ k \right\rvert^2 - 2(n+1)u^*\omega(\Vol_X).  
\end{equation*}
Rewriting and integrating,
\begin{multline*}
 \int_X \left[ \frac{1}{2(n+1)^{\frac{n+1}{2}}}\left\lvert du \right\rvert^{n+1} +  \frac{(n+1)^{\frac{n-3}{2}}}{2\lvert du \rvert^{n-1}}\left\lvert \frac{1}{n} J \circ \Lambda^n du \circ k \right\rvert^2 \right] \dVol_X\\
  = \frac{(n+1)^{\frac{n-3}{2}}}{2}\int_X \frac{\lvert \eth u \rvert^2}{\left\lvert du \right\rvert^{n-1}} \dVol_X + \int_X u^*\omega.
\end{multline*}
\end{proof}


This plurality of energies are not as independent as one might think.  By similar calculations we have 
\begin{Prop}\label{ImpProp1}
For any pointwise orientation non-reversing map $u\colon X \rightarrow M$, the difference 
\begin{equation*} 
E_{n+1}(u) - \int_X u^*\omega  = \frac{(n+1)^{\frac{n-3}{2}}}{2}\int_X g \otimes g_X^*(du,\eth u)\dVol_X,
\end{equation*}
is non-negative.  This difference is $0$ if and only if $\eth u=0$.
\end{Prop}

First a lemma,
\begin{Lem}\label{ImpLemma}
 Let $\zeta_1, \dotsc, \zeta_n$ be an orthonormal $n$-frame in $TM$, and $\zeta_0$ a unit vector.  Then $\omega(\zeta_0 \wedge \dotsc \wedge \zeta_n) = 1$ if and only if $\zeta_0 = J(\zeta_1 \wedge \dotsc \wedge \zeta_n)$. 
\end{Lem}
\begin{proof}
Since $\omega$ is a calibration, we have 
\begin{equation*}    
\omega(\zeta_0\wedge \dotsc \wedge \zeta_n) \leq \lvert \zeta_0 \wedge \dotsc \wedge \zeta_n \rvert \leq 1.                                          
\end{equation*}
If $\zeta_0 = J(\zeta_1 \wedge \dotsc \wedge \zeta_n)$ then the compatibility conditions imply
\begin{equation*}
\omega(\zeta_0 \wedge \dotsc \wedge \zeta_n) = g(J(\zeta_1 \wedge \dotsc \wedge \zeta_n),J(\zeta_1 \wedge \dotsc \wedge \zeta_n)) = \lvert \zeta_1 \wedge \dotsc \wedge \zeta_n \rvert^2 = 1.
\end{equation*}
On the other hand, if $\eta$ is a unit vector for which $\omega(\eta\wedge \zeta_1 \wedge \dotsc \wedge \zeta_n) = 1$, then
\begin{equation*}
 g(\eta, J(\zeta_1 \wedge \dotsc \wedge \zeta_n))=1.
\end{equation*}
Hence $\eta=J(\zeta_1 \wedge \dotsc \wedge \zeta_n)$.
\end{proof}

\begin{proof}(of proposition)
We can look at the integrand pointwise.  Let $\{e_0, \dotsc, e_n\}$ be an orthonormal frame in $T_xX$.  We consider the Hadamard inequality referenced in (\cite{IwaniecMartin01}, 9.9)
\begin{equation*}
 \frac{1}{(n+1)^{\frac{n+1}{2}}}\lvert du \rvert^{n+1} \geq \lvert du \wedge \dotsc \wedge du \rvert.
\end{equation*}
Equality occurs if and only if $du$ is a scalar multiple of an isometry.
Since the dimension of $TX$ is $n+1$, the latter is merely
\begin{equation*}
 \lvert du(e_0) \wedge \dotsc \wedge du(e_n) \rvert.
\end{equation*}
Now since $\omega$ is a calibration,
\begin{equation*}
 u^*\omega(e_0 \wedge \dotsc \wedge e_n) = \omega(du(e_0)\wedge \dotsc \wedge du(e_n)) \leq \lvert du(e_0) \wedge \dotsc \wedge du(e_n) \rvert,
\end{equation*}
which proves the non-negativity of the integrand
\begin{equation*}
 \frac{1}{(n+1)^{\frac{n+1}{2}}}\lvert du \rvert^{n+1} - u^*\omega(e_0 \wedge \dotsc \wedge e_n).
\end{equation*}
If the integrand vanishes we have
\begin{equation*}
 u^*\omega(e_0\wedge \dotsc \wedge e_n) = \lvert du(e_0) \wedge \dotsc \wedge du(e_n) \rvert = \frac{1}{(n+1)^{\frac{n+1}{2}}}\lvert du \rvert^{n+1}.
\end{equation*}
The equality implies $du = \lambda O$ for some scalar $\lambda$ and isometry $O$.  It is clear that $\lambda = \frac{1}{(n+1)^{\frac12}}\lvert du \rvert$.  Then,
\begin{equation*}
 \omega(du(e_0) \wedge \dotsc \wedge du(e_n)) = \lambda^{n+1} \omega(O(e_0) \wedge \dotsc \wedge O(e_n)) = \lambda^{n+1}.
\end{equation*}
The lemma implies
\begin{equation*}
 O(e_0) = J(O(e_1)\wedge \dotsc \wedge O(e_n))
\end{equation*}
along with cyclic permutations.  Multiplying both sides by $\lambda^{n}$ we get,
\begin{equation*}
 \frac{1}{(n+1)^{\frac{n-1}{2}}}\lvert du \rvert^{n-1} du(e_0) = J(du(e_1)\wedge \dotsc \wedge du(e_n)),
\end{equation*}
along with cyclic permutations.  Hence, $\eth u = 0$.

\end{proof}
The results of these calculations are summarized as follows:
\begin{Thm}{$(n+1)$-Energy-Identity}\label{EI}
 A pointwise orientation non-reversing map $u\colon X \rightarrow M$ satisfies the energy identity
\begin{equation}
 E_{n+1}(u) = \frac{(n+1)^{\frac{n-3}{2}}}{2}\int_X g \otimes g^*_X(du,\eth u)\dVol_X + \int_X u^*\omega,
\end{equation}
in which all of the terms are positive.  The map $u$ is multiholomorphic if and only if it satisfies the energy identity
\begin{equation*}
 E_{n+1}(u) = \int_X u^*\omega.
\end{equation*}
\end{Thm}
Hence, we have justified the stance taken in the rest of this paper which is to regard $E_{n+1}$ as \emph{the} relevant notion of energy.  The energy identity immediately implies,
\begin{Cor}
 A pointwise orientation non-reversing map between manifolds with $n$-triads is minimal for the $(n+1)$-energy in its homotopy class if it is multiholomorphic.
\end{Cor}
In the next section we investigate the converse relation to that of the corollary: what constraints does being $E_{n+1}$-minimal place on $u$?  As is the case with harmonic maps of Riemann surfaces into K\"ahler manifolds, in general there \emph{do} exist harmonic maps which are not holomorphic or anti-holomorphic.  
\end{section}

\begin{section}{Variational Aspects}\label{SubSecVariation}
Now, we consider mappings $u \colon X \rightarrow M$ which are in the Sobolev space $W_{loc}^{1,n+1}$, and investigate the Euler-Lagrange equations for the critical points of the $(n+1)$-energy functional
\begin{equation*}
 E(u) = \frac{1}{(n+1)^{\frac{n+1}{2}}}\int_X \lvert du \rvert^{n+1}\dVol_X.
\end{equation*}
$W^{1,n+1}(X,M)$ is defined by using the metrics on $X$ and $M$, and would not be well-defined in their absence due to the degree of integrability being equal to the dimension of the domain.

We work out this calculation in the most general situation after giving a definition of the $p$-Laplacian.  Let $u$ denote a sufficiently differentiable map 
\begin{equation*}
u\colon (X,g_X,(j,k),\dVol_X) \rightarrow (M,g,(J,K),\omega).
\end{equation*}
Because there are metrics on $\Lambda^kT^*X$ and $u^*TM$, there is a metric on the tensor product, $\Lambda^kg^*_X \otimes g$, which we  denote by $\langle ,\rangle$.  Let $\nabla^X$ be the Levi-Civita connection on $X$ and $\nabla^M$ the Levi-Civita connection on $M$.  Let $\nabla^{X*}$ be the dual connection on $T^*X$, that is,
\begin{equation*}
 \nabla^{X*} := g_X\nabla^Xg_X^{-1}.
\end{equation*}
The tensor product connection given by 
\begin{equation*}
\nabla := \nabla^{X*} \otimes I + I \otimes u^*\nabla^M
\end{equation*}
is compatible with $\langle ,\rangle $ in the sense,
\begin{equation*}
 d\langle A,B\rangle  = \langle (\nabla^{X*} \otimes I + I \otimes u^*\nabla^M)A,B\rangle  +  \langle A,(\nabla^{X*} \otimes I + I \otimes u^*\nabla^M)B\rangle. 
\end{equation*}
And so,
\begin{equation*}
 d(\langle du,du\rangle )^{\frac{p}{2}} = p\langle du,du\rangle ^\frac{p-2}{2}\langle \nabla du,du\rangle = p \lvert du \rvert^{p-2}\langle \nabla du, du\rangle. 
\end{equation*} 

\begin{Defn}Given the setup in the previous calculations, with $\nabla^*$ denoting the formal adjoint to the connection $\nabla \colon \Omega^0(u^*TM) \rightarrow \Omega^1(u^*TM)$ with respect to the Hilbert--Schmidt metric, the operator 
\begin{equation*}
 \Delta_p u := \nabla^*(\lvert du \rvert^{p-2}du)
\end{equation*}
is called the \altemph{$p$-Laplacian}.
\end{Defn}

\begin{Thm}
Suppose $u\colon X \rightarrow M$ is a multiholomorphic map with all the usual notation.  Then, $u$ is a solution to an $(n+1)$-Poisson equation, $\Delta_{n+1}u = \eta$ with inhomogeneity depending on the torsion of $J$, 
\begin{equation*}
 \eta := \frac{(n+1)^{\frac{n-1}{2}}}{n}\nabla(J) \circ \Lambda^n du \circ k.
\end{equation*}
\end{Thm}
\begin{proof}
We use the same notation as the previous calculation.  For any smooth $0$-form $\phi$ valued in $u^*TM$ we have
\begin{equation*}
 0 = \int_X \langle \nabla \phi, \eth u\rangle \dVol_X = \int_X \langle \phi, \nabla^* \eth u\rangle \dVol_X.
\end{equation*}
So $u$ weakly satisfies the second-order equation
\begin{equation*}
 (\nabla^* \circ \eth u) = \frac{1}{(n+1)^{\frac{n-1}{2}}}\nabla^*(\lvert du \rvert^{n-1}du)- \frac{(-1)^n}{n}\nabla^*(J \circ \Lambda^n du \circ k)= 0. 
\end{equation*}
It now suffices to show that the second part contributes to the inhomogeneity.  The first part is proportional to the $(n+1)$-Laplacian of $u$, 
\begin{equation*}
 \Delta_{n+1}u = \nabla^*(\lvert du \rvert^{n-1} du).
\end{equation*}
Since $\nabla^*= (-1)^n\star \nabla \star = (-1)^n (\circ j) \nabla (\circ k)$,
\begin{align*}
 \nabla^*(J\circ \Lambda^n du \circ k) =& \nabla(J \circ \Lambda^n du)\circ k,\\
  =& \nabla(J) \circ \Lambda^n du \circ k + J \circ \nabla(\Lambda^n du) \circ k\\ =& \nabla(J) \circ \Lambda^n du \circ k.  
\end{align*}
The last equality follows from the fact that $\nabla(\Lambda^n du)=0$.  Indeed, in local normal coordinates on $X$ centered at some point $z$, because $\nabla^M$ is torsion-free, and the metric is diagonal at $z$, and the Christoffel symbols of $g_X$ vanish there,
\begin{equation*}
 \nabla^M_i \partial_ju = \nabla^M_j \partial_i u, \,\,\, \text{and} \,\,\, \nabla du = (I \otimes \nabla^M)(du).  
\end{equation*}
And so, by the combinatorics of alternating products,
\begin{align*}
 &\sum_i \nabla^M_i J(\partial_{i+1} u \wedge \dotsc \wedge \partial_{i-1} u) \cdot dx^i \wedge dx^{i+1}\wedge \dotsc \wedge dx^{i-1} \\
 &= \sum_{i} \sum_{j \neq i} J(\partial_{i+1} u \wedge \dotsc \wedge \nabla^M_i \partial_j u \wedge \dotsc \wedge \partial_{i-1}u) dx^i \wedge \dotsc \wedge dx^{i-1} = 0.
\end{align*}
This computation applies at the origin point for the normal coordinates on $X$.  However, since this is a tensorial equation, the identity $\nabla(\Lambda^n du)=0$ is true globally.
\end{proof}
Most of the time in this paper we assume the vector cross product $J$ to be torsion-free, and hence a multiholomorphic map is $(n+1)$-harmonic.  The next proposition is a standard fact.
\begin{Prop}
On $\mathcal{B} := W^{1,n+1}(X,M)$ the $(n+1)$-harmonic maps extremize the functional $E\colon \mathcal{B} \rightarrow \R$.
\end{Prop}
\begin{proof}
 Let $\gamma \colon \R \rightarrow \mathcal{B}$ be a path $\lambda \mapsto u_\lambda$ such that $\gamma(0) = u$, and $\dot{\gamma}(0) = \zeta$.  We have,
\begin{align*}
 dE(u)\zeta =& \frac{d}{d\lambda}\left[E(u_\lambda)\right]\rvert_{0} = \int_X \frac{d}{d\lambda}\left[\langle du_\lambda,du_\lambda\rangle ^{\frac{n+1}{2}}\dVol_X\right]\rvert_0 \dVol_X= \\
 =& (n+1)\int_X\lvert du \rvert^{n-1} \langle \nabla_\lambda du_\lambda,du_\lambda\rangle \rvert_0 \dVol_X\\
 =& (n+1)\int_X \langle \nabla \zeta,\lvert du \rvert^{n-1}du\rangle \dVol_X 
\end{align*}
This differential is zero for all $\zeta$ if and only if $\Delta_{n+1}u = 0.$
\end{proof}

In $\R^{n+1}$ it has been shown in the work of Uhlenbeck \cite{Uhlenbeck77} and Evans \cite{Evans82}, and related work in \cite{Tolksdorf84}, that \emph{a priori} a solution to the $(n+1)$-Laplace equation on real functions (with constant coefficients) on $\R^{n+1}$ must have $C^{1,\alpha}$-regularity.  In principle this regularity is the best one can do as a result of the degeneracy of $\Delta_p$.  In the general Riemannian setting the estimates of Hardt and Lin \cite{HardtLin87}, cf. White \cite{White88} for minimizers of the $(n+1)$-energy suffice to prove $C^{1,\alpha}$ interior regularity, and smooth regularity off of a Hausdorff codimension-$2$ locus.  However it is the hope that with some \emph{a priori} control on the critical locus of $u$ this situation improves.  Along these lines one might want to consider the \emph{a priori} ``size'' of the zero locus of $du$ --the topic of section \ref{SubSecCrit}.

To finish off the section, we mention an important theorem pertaining to minimizers of the $p$-energy.
Wei has proven an existence theorem for $p$-energy minimizers,
\begin{Thm}{(\cite{Wei94})} \label{Existence}
 Let $u \colon X \rightarrow M$ be a continuous map of finite $p$-energy.  Assume $M$ has non-positive sectional curvature.  Then, $u$ is homotopic to a $C^{1,\alpha}$ $p$-minimal map.  
\end{Thm}
This result would give an existence theorem for (anti-)multiholomorphic maps whenever it is the case that all $p$-energy minimizers are (anti-)multiholomorphic.  Even in the harmonic situation this is rarely the case.


\end{section}
\begin{section}{Endomorphisms and Quasiregular Mappings} \label{SubSecEndo}
The multi-Cauchy--Riemann equations in Euclidean space can be seen to coincide with a system of equations called the Cauchy--Riemann system.  Let $\Omega$ be an open domain in $\R^k$, $f\colon \Omega \rightarrow \R^k$ a $W_{loc}^{1,k}$ map, and $\lvert J(f) \rvert $ the Jacobian determinant of $f$.  
\begin{Defn}
 The \altemph{Cauchy--Riemann (CR)} system in $\R^k$ is the system of equations
\begin{equation}\label{CRsys}
 Df^t \circ Df = \lvert J(f) \rvert^{\frac{2}{k}}\cdot I, \,\,\, \mathrm{a.e.}
\end{equation}
\end{Defn}
This system is an overdetermined (if $k \geq 3$), first-order, fully non-linear system.  It can be extended to the Riemannian setup as follows.
\begin{Defn}
 Let, $(M,g_M,\dVol_M,(-1)^n\star), (N,g_N,\dVol_N,(-1)^n\star)$ be $(n+1)$-dimensional, oriented, Riemannian manifolds with conformal triads, and $u$ a map between them.  Denote by $\Vol_M$ the oriented, unit, field of $(n+1)$-vectors on $M$ dual to $\dVol_M$.  The \altemph{Cauchy--Riemann system} is the equation
\begin{equation}
 u^*g_N = \left[u^*\dVol_N(\Vol_M)\right]^{\frac{2}{n+1}} \cdot g_M. 
\end{equation}
\end{Defn}

The motivating theorem in this field is the classical Liouville theorem for $C^3$ functions \citep{IwaniecMartin01}{2.3.1}
\begin{Thm}{(Liouville)}
 Let $\Omega \subsetneq \R^n$ be an open domain.  Every solution $f \in C^{3}(\Omega,\R^n), n \geq 3,$ to the CR-system eqn. (\ref{CRsys}), where $\lvert J(f) \rvert$ does not change sign in $\Omega$ is of the form
\begin{equation*}
 f(x) = b + \frac{\alpha A(x-a)}{\lvert x-a \rvert^\epsilon},
\end{equation*}
for some $a \in \R^n \setminus \Omega, b \in \R^n, \alpha \in \R, A \in O(n), $ and $\epsilon = 0$ or $2$.
\end{Thm}
This theorem is a strong rigidity result which shows that the $n=2$ case (in which the space of solutions is infinite dimensional) is exceptional.

This Liouville Theorem was extended to the case of functions of Sobolev class $W_{loc}^{1,n}$ by the labors of Ghering, Reshetnyak, and by different methods, Bojarski-Iwaniec, and Iwaniec-Martin (see \cite{IwaniecMartin01} pg. 85 for these references). 

From a geometric perspective, solutions of equation (\ref{CRsys}) are maps of zero distortion in the sense that their directional derivatives satisfy 
\begin{equation*}
 \text{max}_{\lvert \alpha \rvert = 1} \lvert \partial_\alpha f(x) \rvert  =  \text{min}_{\lvert \alpha \rvert=1}\lvert \partial_\alpha f(x) \rvert, \,\,\, \mathrm{a.e.}.
\end{equation*}
Therefor these solutions are, in particular, maps with bounded distortion.
\begin{Defn}
 A mapping $f \colon \Omega \subset \R^n \rightarrow \Omega^\prime \subset \R^n$ in $W_{loc}^{1,n}$ has \altemph{$Q$-bounded distortion} for some $Q \geq 1$ if its distortion,
\begin{equation*}
 \frac{\text{max}_{\lvert \alpha \rvert=1} \lvert \partial_\alpha f(x) \rvert}{\text{min}_{\lvert \alpha \rvert=1}\lvert \partial_\alpha f(x) \rvert}
\end{equation*}
is finite and bounded by $Q$ almost everywhere.
\end{Defn}
The definition of $Q$-bounded distortion extends to the Riemannian scenario by replacing the Euclidean norms with the metric norms.  

A stronger condition that compares the length of derivative vectors to volumes of tangent-planes is the notion of $Q$-quasiregularity.
\begin{Defn}\label{qrDef}
 A mapping $f \colon \Omega \subset \R^n \rightarrow \Omega^\prime \subset \R^n$ in $W_{loc}^{1,n}\cap C^{0}$ is \altemph{$Q$-quasiregular} for some $\infty > Q \geq 1$ if,
\begin{equation*}
 \frac{1}{n^{\frac{n}{2}}}\lvert df \rvert^{n} \leq Q  J(f), \,\,\, \mathrm{a.e.}
\end{equation*}
\end{Defn}
Again, the definition of $Q$-quasiregular can be immediately extended to the Riemannian setting with $J(f)$ replaced by $f^*\mathrm{dVol_M}(\mathrm{Vol_N})$.  With this notion in hand, it is easy enough to see that a solution to the MCR equation has $1$-bounded distortion (i.e. no distortion) and is $1$-quasiregular.  Consider the following pointwise calculations on $u\colon N \rightarrow M.$  Let $\zeta$, $\eta$ be unit tangent vectors at some point of $N$, and complete these to oriented orthonormal frames $\{ \zeta_0=\zeta,\dotsc,\zeta_n \}$, $\{ \eta_0=\eta,\dotsc,\eta_n\}$, so that $\star \zeta = \zeta_1 \wedge \dotsc \wedge \zeta_n$.  Then,
\begin{align*}
 \lvert du \rvert^{n-1} g_M(du(\zeta),du(\zeta)) &= \frac{(n+1)^{\frac{n-1}{2}}}{n}g_M(du(\zeta),\star \Lambda^n du( \star \zeta)) = (n+1)^{\frac{n-1}{2}}u^*\dVol_M(\zeta_0\wedge \dotsc \wedge \zeta_n)\\
 &= (n+1)^{\frac{n-1}{2}}u^*\dVol_M(\eta_0 \wedge \dotsc \wedge \eta_n) = \lvert du \rvert^{n-1}g_M(du (\eta),du (\eta)).
\end{align*}
Hence the distortion is the same in arbitrary directions $\eta, \zeta$.  The calculation in lemma \ref{AutomLem} establishes $1$-quasiregularity, so we don't repeat it.

\begin{Prop}
 Given $(M,g_M,\dVol_M,(-1)^n\star)$, $(N,g_N,\dVol_N,(-1)^n\star)$, two $(n+1)$-dimensional Riemannian manifolds with conformal triads, a map $u \colon M \rightarrow N$ is a solution of the multi-Cauchy Riemann equations if and only if it is a solution to the Cauchy--Riemann system. 
\end{Prop}
\begin{proof}
First of all, recall the compatibility conditions on the conformal triads.  In particular,
\begin{equation*}
 \dVol_N(\zeta \wedge \star \eta) = g_N(\zeta,\eta).
\end{equation*}
Given $u$ a mulitholomorphic map,
\begin{align*}
 g_N(du(\zeta),du(\eta)) &= \dVol_N(du(\zeta)\wedge \star du(\eta)) = \frac{(n+1)^{\frac{n-1}{2}}}{n \lvert du \rvert^{n-1}}\dVol_N(du(\zeta)\wedge \Lambda^n du(\star \eta))\\
 &= \frac{(n+1)^{\frac{n-1}{2}}}{\lvert du \rvert^{n-1}}u^*\dVol_N(\zeta \wedge \star \eta) = \frac{(n+1)^{\frac{n-1}{2}}}{\lvert du \rvert^{n-1}}u^*\dVol_N(\Vol_M)g_M(\zeta,\eta).
\end{align*}
Now by applying $g_N \otimes g_M^* (du,\bullet)$ to both sides of the MCR equation (as per the energy identity (thm. \ref{EI})) it is easy to show pointwise 
\begin{equation*}
 \frac{1}{(n+1)^{\frac{n+1}{2}}}\lvert du \rvert^{n+1} = u^*\dVol_N(\Vol_M),
\end{equation*}
so that we get the Cauchy--Riemann system
\begin{equation*}
 u^*g_N = \left[u^*\dVol_N(\Vol_M)\right]^{\frac{2}{n+1}}g_M.
\end{equation*}
The other implication follows by argumentation along the lines of proposition \ref{ImpProp1}.  What follows is a summary: the CR-equation implies 
\begin{equation*}
 \frac{1}{(n+1)^{\frac{n+1}{2}}}\lvert du \rvert^{n+1} = u^*\dVol_N(\Vol_M),
\end{equation*}
which is the extremal case of the Hadamard inequality (see proposition \ref{ImpProp1}).  The extremal case implies that pointwise
\begin{equation*}
du = \frac{1}{(n+1)^{\frac{1}{2}}} O,
\end{equation*}
for some isometry $O$.  And finally by the properties of the vector cross product $J$,
$\eth u = 0$.  (see lemma \ref{ImpLemma})
\end{proof}

\begin{Example}
Any Riemannian covering which is locally a conformal equivalence is a multiholomorphic map.   

\end{Example}
%
%
%

Reshetnyak \cite{Reshetnyak67} established the following strong rigidity result for solutions of the Cauchy--Riemann system in $\R^n$,
\begin{Thm}(Reshetnyak)
 Given $u \in W^{1,n}(\Omega,\R^n)$ satisfying the Euclidean CR-system, then $u$ is discrete and open.  That is, the preimage of a point is a discrete set, and $u$ maps open sets to open sets.   
\end{Thm}

These kinds of rigidity theorems are relevant to general multholomorphic mappings because the space of such maps $u\colon X \rightarrow M$ will locally be foliated by orbits of the multiholomorphic automorphism group of $X$, that is, solutions of the afore-mentioned multiholomorphic system on maps $\phi \colon X \rightarrow X$.  

If $X$ is locally conformally flat then the results on maps satisfying the CR-system in $\R^{n+1}$ apply directly to the problem of characterizing multiholomorphic (endo)automorphisms of $X$.  But in general those results would need to be extended to the case of non-vanishing conformal curvature/Weyl tensor on $X$.  To the author's knowledge little in known in generality in this direction, but there are some results dealing with uniformly quasiregular maps.  It is shown, for instance, in \cite{BirdsonHinkkanenMartin} that a closed hyperbolic manifold cannot admit a non-trivial uniformly quasiregular self map at all.  For our purposes such an extension is unnecessary as we will be able to put the Euclidean theorems to good use in the next section.

\end{section}

\begin{section}{Concerning the Critical Locus of a Multiholomorphic Map}\label{SubSecCrit}

As already mentioned in section \ref{SubSecEndo}, Reshetnyak managed to prove that under the relatively weak regularity/integrability assumptions of $u \in W^{1,n+1}(\Omega \subset \R^{n+1},\R^{n+1})$, a multiholomorphic endomorphism of a Euclidean domain (a.k.a. a solution to the CR-system) is open and discrete.  He leveraged the critical fact that the function
\begin{equation*}
 w = -\log \lvert u \rvert
\end{equation*}
turns out to be $(n+1)$-harmonic on the complement of the zero-locus of $u$.  From the comparison principle for such functions, one can estimate (using conformal capacity theory) the size of the zero-locus of $u$, and ultimately conclude that it is discrete (\cite{IwaniecMartin01}, Ch. 16.). 

When $n \geq 2$ there is strong rigidity in the presence of higher regularity.  In the following theorem fractional regularity is understood in terms of the H\"older-continuity in the sense 
\begin{equation*}
 C^p := C^{\lfloor p\rfloor,p-\lfloor p \rfloor}.
\end{equation*}
 
\begin{Thm}{\cite{BonkHeinonen}}
If $n \geq 2$, every non-constant, pointwise orientation-non-reversing solution $f \in C^{\frac{n+1}{n-1}}(\Omega \subset \R^{n+1},\R^{n+1})$ to the Euclidean CR-system (more generally a uniformly quasiregular map) is a local homeomorphism.
\end{Thm}

Alternately, if one is willing to trade a specific, small upper bound on the distortion for less regularity, there is another relevant local rigidity result.  Let $\Omega$ be a open domain in $\R^n$, and $f$ a quasiregular map into $\R^n$.  Then, the \emph{outer dilation} of $f$, $K_O(f)$ is defined
\begin{equation*}
 K_O(f) := \mathrm{inf}_{K > 1} \left \{ K \Big \vert \quad \frac{1}{n^{\frac n2}}\lvert df \rvert^n \leq K J(f) \quad \mathrm{a.e.} \right\}.
\end{equation*}
  
\begin{Thm}{\cite{MartioRickmanVaisala71},\cite{Goldshtein71}, c/f \citep{Srebro92}{2.3, (ii)}}\label{rigidity}
There exists a universal constant $K>1$ such that every quasiregular mapping $f \colon \Omega \rightarrow \R^n,\,\, n\geq 3$ with $K_O(f) \leq K$ is a local homeomorphism. 
\end{Thm}

\begin{Lem}
Let $u \colon X \rightarrow X$ be a multiholomorphic endomorphism of an $n+1$-dimensional Riemannian manifold $X$ with conformal triad.  Then $u$ is locally quasiregular in the sense that every point in $X$ has a coordinate neighborhood on which $u$ becomes quasiregular with respect to the Euclidean metrics. 
\end{Lem}
\begin{proof}
Throughout this proof we work under the assumption that $u$ is $C^1$ in order to keep notation clear.  This is not a substantive restriction; one must merely write $\mathrm{a.e.}$ and $\mathrm{esssup}$ where applicable.
 
Since $u$ is $1$-quasiregular in the Riemannian sense with metrics $g,g^\prime$, in particular it has constant distortion $1$.  That is, 
\begin{equation*}
 \mathrm{max}_{\alpha} \frac{\lvert du(\alpha) \rvert_{g^\prime}}{\lvert \alpha \rvert_g} =  \mathrm{min}_{\alpha} \frac{\lvert du(\alpha) \rvert_{g^\prime}}{\lvert \alpha \rvert_g}
\end{equation*}
at regular points of $u$.  

We consider a bounded local normal coordinate chart on $X$ small enough that its image under $u$ falls inside a local normal coordinate chart on $X$ centered at $u(0)$.  Hence, w.l.o.g. we can regard $u \colon \Omega \subset \R^{n+1} \rightarrow \R^{n+1}$ with $u(0)=0$.  
The distortion between the domain metric $g$ and the Euclidean metric on this chart can be measured by functions
\begin{equation*}
 S(x):= \max_{\alpha \neq 0 \in \R^{n+1}} \frac{\lvert \alpha \rvert_0}{\lvert \alpha \rvert_g}
\end{equation*}
\begin{equation*}
 I(x):= \min_{\alpha \neq 0 \in \R^{n+1}} \frac{\lvert \alpha \rvert_0}{\lvert \alpha \rvert_g}.
\end{equation*}
Furthermore, the coordinate domains can always be chosen such that these distortion functions are bounded.  After all, in such coordinates $g$ is Euclidean to second order in the radius which means, asymptotically for any $\alpha$ as above, 
\begin{equation*}
 1-o(r^2) \leq \frac{\lvert \alpha \rvert_0}{\lvert \alpha \rvert_g} \leq 1 + o(r^2).
\end{equation*}
So there exist finite bounds
\begin{equation*}
 S := \sup_{x} S(x) \geq 1, \quad I := \inf_x I(x) \leq 1.
\end{equation*}
Let $S^\prime$ and $I^\prime$ be the corresponding bounds in the coordinate patch in the image.  Then the hypothesized distortion inequality (at an arbitrary regular point of the domain)
\begin{equation*}
 \max_{\alpha} \frac{\lvert du(\alpha) \rvert_{g^\prime}}{\lvert \alpha \rvert_g} = \min_{\alpha} \frac{\lvert du(\alpha) \rvert_{g^\prime}}{\lvert \alpha \rvert_g}
\end{equation*}
leads to 
\begin{equation*}
 \max_{\alpha} \frac{\lvert du(\alpha)}{\lvert \alpha \rvert_0} \leq \frac{S \cdot S^\prime}{I \cdot I^\prime}\min_{\alpha} \frac{\lvert du(\alpha) \rvert_{g^\prime}}{\lvert \alpha \rvert_g}.
\end{equation*}
Hence, $u$ is locally of $\frac{S \cdot S^\prime}{I \cdot I^\prime}$-bounded distortion.

Similarly we measure the distortion of the volume of $n+1$-planes between the given metrics and the Euclidean metrics by functions
\begin{equation*}
 R(x):= \max_{A \neq 0 \in \Lambda^{n+1}\R^{n+1}} \frac{\lvert A \rvert_0}{\lvert A \rvert_g}
\end{equation*}
\begin{equation*}
 Q(x):= \min_{A \neq 0 \in \Lambda^{n+1}\R^{n+1}} \frac{\lvert A \rvert_0}{\lvert A \rvert_g}.
\end{equation*}
Furthermore, the coordinate domains can always be chosen small enough such that these volume-distortion functions are bounded.  After all, in such coordinates $g$ is Euclidean to second order in the radius which means, the induced metric on the higher exterior powers of $R^{n+1}$ are also asymptotically Euclidean to at least second order.  

So there exist finite bounds
\begin{equation*}
 R := \sup_{x} R(x) \geq 1, \quad Q := \inf_x Q(x) \leq 1.
\end{equation*}
Let $R^\prime$ and $Q^\prime$ be the corresponding bounds in the coordinate patch in the image.

Then, we estimate from above the norm $\lvert du \wedge \dotsc \wedge du \rvert$:
\begin{equation*}
 \lvert du \wedge \dotsc \wedge du \rvert := \frac{\lvert du(e_0) \wedge \dotsc \wedge du(e_n) \rvert}{\lvert e_0 \wedge \dotsc \wedge e_n \rvert} \leq \frac{R}{Q^\prime}\frac{\lvert du(e_0) \wedge \dotsc \wedge du(e_n)\rvert_0}{\lvert e_0 \wedge \dotsc \wedge e_n \rvert_0} =\frac{R}{Q^\prime} \lvert du \wedge \dotsc \wedge du \rvert_0.
\end{equation*}
From below we estimate the norm $\lvert du \rvert$, where $\{v_i\}$ is a $g$-orthogonal frame.
\begin{equation*}
 \lvert du \rvert := \sum_{v_i} \frac{\lvert du(v_i) \rvert}{\lvert v_i \rvert}
\end{equation*}
Because $du$ pointwise is a $g$-$g^\prime$-orthogonal transformation, we can replace the above sum with the Euclidean coordinate frame 
\begin{equation*}
 \lvert du \rvert = \sum_{e_i}  \frac{\lvert du(e_i) \rvert}{\lvert e_i \rvert}.
\end{equation*}
By the distortion estimates already computed, we get a lower bound on $\lvert du \rvert$,
\begin{equation*}
 \frac{I}{S^\prime}\lvert du \rvert_0 \leq \lvert du \rvert.
\end{equation*}
The definition of $1$-quasiregularity specifies
\begin{equation*}
 \frac{1}{(n+1)^{\frac{n+1}{2}}}\lvert du \rvert^{n+1} = \lvert du \wedge \dotsc \wedge du \rvert.
\end{equation*}
So we have the desired estimate in terms of the Euclidean metrics:
\begin{equation*}
 \frac{1}{(n+1)^{\frac{n+1}{2}}}\lvert du \rvert_0^{n+1} \leq \frac{RS^\prime}{IQ^\prime}\lvert du \wedge \dotsc \wedge du \rvert_0.
\end{equation*}
\end{proof}

\begin{Cor}(Of Thm (\ref{rigidity}) and the lemma)\label{Cor1}
 Let $X^{n+1}$ be a compact Riemannian manifold with $n \geq 2$.  Then a non-constant multiholomorphic endomorphism of class $W^{1,n+1} \cap C^0$ is a local homeomorphism, and hence is an orientation-preserving conformal covering.
\end{Cor}
\begin{proof}
The previous lemma implies that any multiholomorphic endomorphism is locally quasiregular for some dilation bound $K$ depending on the particular local normal coordinate patch chosen.  The patch in the domain can be chosen to be arbitrarily small, and since the Riemannian metric is asymptotic to the Euclidean one, the distortion bound on the patch must get arbitrarily close to $1$ as the radius is reduced.  This means that every point in $X$ has a coordinate neighborhood (and a neighborhood centered at it's image), with arbitrarily small distortion bound $K \geq 1$.  By the rigidity theorem (\ref{rigidity}), this means the function is a local homeomorphism.  
\end{proof}

We have demonstrated that multiholomorphic endomorphisms cannot have branch loci if the dimension of the domain is greater than two.  A similar question could be raised about general multiholomorphic maps.  With a view towards this issue we demonstrate a unique continuation theorem for smooth multiholomorphic maps.  This result has the effect that it characterizes the critical locus of a smooth multiholomorphic map as a finite collection of points.  For regularity $C^3$, the best one can say by the same methods is that the critical locus will have Hausdorff codimension at least two.

The main observation to leverage is that any multiholomorphic map is weakly conformal (see eqn. (\ref{weakconformal})).  Pan, \cite{Pan95}, has established a unique continuation theorem for weakly conformal maps between Riemannian manifolds of the same dimension which can be extended without much work to the current situation.  As a result we can state the following theorem.

\begin{Thm}\label{uniquecontinuation}
 Let $u \colon X \rightarrow M$ be a multiholomorphic map of regularity at least $C^3$.  Then, if $u$ vanishes to infinite order at any point it is constant.  Furthermore, if $u$ is smooth, the critical locus of $u$ consists of a finite collection of points in $X$.  
\end{Thm}

The proof of the theorem ---which is essentially the main result of \cite{Pan95}--- is established by a two-step process.  First we observe that the function $\lambda = \frac{\lvert du \rvert^2}{n+1}$ satisfies a second order, elliptic equation of Laplace type.  This largely follows the calculation in \cite{Pan95} used to establish the estimate $\lvert \Delta \lambda \rvert \leq C \lambda$; we give a invariant translation.  From the first observation one can massage the important result of Aronszajn concerning unique continuation of solutions to second-order elliptic equations.  The result implies that the differential of a multiholomorphic map cannot vanish to infinite order anywhere unless the map is constant.  Furthermore, if the map were smooth to begin with then the zero locus of $\lambda$ cannot have accumulation points.  We organize this argument into two lemmas.

Throughout the rest of this section, any time a local calculation is made, we will use local normal coordinates with tangent frame $\{ e_i \}$ at the origin.  We will use the shorthand $\nabla_i := \nabla_{e_i}$, and $du_i := du(e_i)$.

\begin{Lem}
Let $u: X \rightarrow M$ be a $C^3$ multiholomorphic map.  The function $\lambda := \frac{\lvert du \rvert^2}{n+1}$ satisfies a second-order, equation of Laplace type,
\begin{equation}\label{WeakConfEqn}
 \Delta \lambda = g.
\end{equation}
The non-homogeneity function $g$ satisfies the estimates,
\begin{equation}
 \lvert g \rvert \leq C \lambda.
\end{equation}
\end{Lem}
\begin{proof}
Let $\nabla$ denote the Levi-Civita connection on $(X,g_X)$ as well as the induced connection on any tensors constructed from $TX$.  Let $\nabla^M$ likewise be the Levi-Civita connection on $(M,g)$.  We equip the bundle $T^*X \otimes u^*TM$ over $X$ with a natural connection
\begin{equation*}
 \tilde{\nabla} := \nabla^* \otimes I + I \otimes u^*\nabla^M = -\nabla \otimes I + I \otimes u^*\nabla^M.
\end{equation*}
Note that $du$ is a section of this bundle.  Let $R = \nabla^2 \in \Omega^2(\End(TX))$ be the Riemannian curvature tensor of $g_X$ and $K$ likewise the curvature of $g$ on $M$.  Then we compute the curvature of $\tilde{\nabla}$:
\begin{equation}
 \mathcal{R} = R \otimes I + I \otimes (u^*\nabla^M)^2 = R \otimes I + I \otimes u^*K. 
\end{equation}
The second exterior derivative, $\tilde{\nabla}^2(du)$, then relates to these curvature terms.  We have for any $Z,W,X$ in $TX$, and $\eta$ in $TM$,
\begin{equation} \label{curvature}
 g(\tilde{\nabla}^2_{Z,W}(du)(X),\eta) = g(du(R(Z,W)X),\eta) + g(K(du(Z),du(W))du(X),\eta).
\end{equation}

Abusing notation we write 
\begin{align*}
 K(X,Y,Z,W) :=& g(K(X,Y)Z,W),\\
 S_R :=& \sum_{jk} u^*g(R(e_j,e_k)e_j,e_k),\\
 S_K :=& \sum_{jk} u^*K(e_j,e_k,e_j,e_k).
\end{align*}
In this notation, Pan's calculation amounts to firstly,  
\begin{equation}\label{firstident}
\Delta(\trace_{g_X}u^*g)= 2S_R + 2S_K+ 2\sum_{jk}\left[ g(\tilde{\nabla}_{kj}(du_j) \otimes du_k)\right ]+ 2\lvert \tilde{\nabla}(du) \rvert^2,
\end{equation}
---in which the norm of $\tilde{\nabla}(du)$ is taken in $\Lambda^2 T^*X \otimes u^*TM$--- and secondly to,
\begin{align}\label{secondident}
 (n+1)\Delta\lambda &= (n+1)\sum_{jk}\delta_{jk} \big[ 2g(\tilde{\nabla}_{kj}(du_j) \otimes du_k)\\
 &+ g (\tilde{\nabla}_j(du_k)\otimes\tilde{\nabla}_j(du_k) + g (\tilde{\nabla}_j(du_j)\otimes\tilde{\nabla}_k(du_k) \big ].
\end{align}
Regarding the relation coming immediately from $(\ref{firstident})$,
\begin{equation}
 \Delta \lambda= \frac{2}{n+1}(S_R + S_K)+ \frac{2}{n+1}\sum_{jk}\left[ g(\tilde{\nabla}_{kj}(du_j) \otimes du_k)\right ]+ \frac{2}{n+1}\lvert \tilde{\nabla}(du) \rvert^2,
\end{equation}
 as
 $\Delta \lambda = g$,
we see that $\lvert g \rvert \leq C \lambda$.  Taking the difference $(\ref{secondident})- (\ref{firstident})$ yields
\begin{equation}\label{Pan12}
 (n-1)\left [\lvert \tilde{\nabla}(du)\rvert^2 + \lvert \trace_{g_X}\tilde{\nabla}(du)\rvert^2 + (S_R + S_K)\right] +(2n)\sum_{ij}g(du_j),\tilde{\nabla}^2_{ji}du_i)= 0.
\end{equation}
We apply Cauchy-Schwarz to the norm of $S_R$, and take account that $R$ is bounded to derive
\begin{equation*}
\lvert S_R := \sum_{jk} g(du(R(e_j,e_k)e_j),du_k) \rvert \leq \sum_{jk} \lvert du(R(e_j,e_k)e_j) \rvert \lvert du_k \rvert \leq C \lvert du \rvert^2 = (n+1)C \lambda.
\end{equation*}
The same can be done for $S_K$.  If $n>1$, then (\ref{Pan12}), another application of Cauchy-Schwarz, and the boundedness of $\tilde{\nabla}^2_{ji}du_i$ yields
\begin{align*}
 \left [\lvert \tilde{\nabla}(du)\rvert^2 + \lvert \trace_{g_X}\tilde{\nabla}(du)\rvert^2 \right] &= -(S_R + S_K) +\frac{2n}{n-1}\sum_{ij}g(du_j,\tilde{\nabla}^2_{ji}du_i)\\
 & \leq C(\lvert S_R \rvert + \lvert S_K \rvert) + C \lvert du \rvert^2 \sum_{ij}g(\tilde{\nabla}^2_{ji}du_i,\tilde{\nabla}^2_{ji}du_i)\\
 & \leq C\lambda.
\end{align*}
Combining the above estimate with the fact that $\tilde{\nabla}^2_{ji}du_i$ is bounded, implies that all three terms of the inhomogeneity $g$ are bounded by $C\lambda$ for some appropriate constant $C$.
\end{proof}

Next we cite the theorem of Aronszajn
\begin{Thm}(\cite{Aronszajn57})
Let $L$ be a second-order, uniformly elliptic, linear differential operator on a bounded, open domain in $\R^m$.  Let $f$ be a $\mathcal{C}^2$ function solving
\begin{equation*}
 Lf = g,
\end{equation*}
with
\begin{equation*}
 \lvert g \rvert \leq C(\lvert f \rvert + \lvert \nabla f \rvert).
\end{equation*}
If $f$ vanishes to infinite order at some point $x_0$, i.e.,
\begin{equation*}
 \lim_{r := \lvert x - x_0\rvert \rightarrow 0} \frac{f(x)}{r^n} = 0, \,\,\, \forall n \in \mathbb{N},
\end{equation*}
then $f \equiv 0$.
\end{Thm}

Aronszajn's theorem applies then to $\lambda$.  If $u$ were smooth, then $\lambda$ could not have zeros which accumulate without $\lambda$ vanishing to infinite order, hence $\lambda$'s zeros could not accumulate.  The theorem is proved.

%
%
We have shown that a multiholomorphic map which is smooth can have at most critical points which are isolated.  It remains yet to be discovered whether or not these rigidity properties are sharp in the sense that there are examples of multiholomorphic maps which are not conformal covers of calibrated submanifolds.  

In the case when differentiability is only assumed to be $C^3$, the the unique continuation argument implies that the critical locus cannot be so large that the function vanishes to infinite order.  This condition is not an incredibly strong restriction on the size of the critical locus.  However, by the work of Hardt and Lin \cite{HardtLin87} we know a multiholomorphic map is $C^\infty$ in the compliment of a Hausdorff codimension-$2$ locus.  Hence, we can apply the unique continuation result on such a complement to derive
\begin{Cor}
 Let $u \colon X \rightarrow M$ be a multiholomorphic map of regularity at least $C^3$.  Then, the critical locus of $u$ has Hausdorff codimension at least two.
\end{Cor}

We conclude the section with a speculation about the regularity of multiholomorphic endomorphisms.  Given a multiholomorphic endomorphism $u$ of $X^{n+1}$ in $W^{1,n+1} \cap C^{0}$, we have shown that $u$ is a $n+1$-harmonic map (in fact a stable point for the $n+1$-energy).  \emph{A priori} the $p$-Laplace operator is degenerate elliptic.  The ellipticity fails at critical points of $u$.  Under the assumption that $u$ has empty critical locus, we are left with an elliptic equation, and hence might attempt to bootstrap to higher regularity.  

Suppose $u \colon X \rightarrow X$ satisfies the hypotheses of corollary \ref{Cor1}.  Since $u$ has empty critical locus, $f := \lvert du \rvert^{n-1}$ is a positive function on $X$.  And $u$ is a solution to the divergence-form \emph{elliptic} PDE
\begin{equation}\label{ellipticeqn}
  \nabla^*(f \cdot du) = 0.
\end{equation}   
We would want to improve the coefficients by applying the standard $L^{n+1}$-elliptic regularity results for divergence-form equations to improve the regularity.  However, the standard estimates seem slightly too weak to allow this luxury.  At least the problem appears more sophisticated, and particular than we had hoped.  As a result we leave the following speculation to be (dis)confirmed: 

\begin{Spec}
Any multiholomorphic endomorphism $u \colon X^{n+1} \rightarrow X^{n+1}$ of class $W^{1,n+1} \cap C^0$, with $n \geq 2$ is a smooth, conformal cover.  
\end{Spec}

\end{section}

\begin{section}{Motivation}\label{SubSecMainGoals}
From the start, the general aim of this project has been to study the ``mapping theory'', associated to the framework of triadic structures and multiholomorphic maps ---a clear source of inspiration coming from the various types of symplectic invariants (e.g. Lagrangian/Hamiltonian Floer theory, quantum cohomology) which are constructed from holomorphic curves.  The data coming from the moduli of such maps could give information about the geometry of the target manifold or potentially about the singularities of the map images (which generally are calibrated objects).

The starting data includes a compact, $n$-triadic manifold $(M,\omega,g,(J,K))$, as well as any compact, $(n+1)$-dimensional Riemannian manifold $X$ equipped with the associated conformal triad.  $X$ is regarded as the domain and $M$ the target, and the main object of study will be configurations of multiholomorphic maps from $X$ to $M$.
\begin{equation*}
 \{ \text{$(n+1)$-manifolds with conformal triad} \} \xrightarrow{\eth u = 0} \{ \text{ manifolds with $n$-triad} \}
\end{equation*}\\
It should not come as a surprise at this point that we will really only need the data of conformal transformation equivalence class of $X$, so we expect that any relevant invariants will be conformal transformation invariants.  One then considers the moduli spaces of multiholomorphic maps from $X$ into $M$, a space which is split into components (by virtue of the energy identity) by the possible topological classes of such maps.  We might also consider boundary conditions for non-closed $X$ which lie on ``branes,'' that is, maximal isotropic submanifolds for the multiholomorphic form $\omega$.  For example, in the $\Gtwo$-case the branes are the coassociative submanifolds.  There is an obvious choice for the branes for each of the distinct families of n-triad structures.  There is preexisting work along these lines in \cite{LeeLeung08},\cite{LeeLeung09},\cite{Leung02}, (and other work of Leung), which in turn rests on the work of McLean \cite{McLean98} with regard to the deformation theory of calibrated submanifolds and the associated branes.    

The main novelty of this paper's framework is a PDE intertwining compatible structures on domain and target.  Prior work seems to be concerned mainly with arbitrary parametrizations of associative submanifolds, a class of maps which is much larger and does not benefit from the energy results enjoyed by solutions of the MCR equation.  The primary trade off seems to be the evident strong rigidity ---e.g. Thm \ref{uniquecontinuation} and the $\Gtwo$-Liouville theorem in the last section.  One can also compute that the deformation problem for multiholomorphic maps is overdetermined in dimension greater than two.  One might regard a multiholomorphic map as a special parametrization of a calibrated submanifold (or current).  



For most of what follows we will be concerned with a particular realization of the multiholomorphic framework ---the first case in which the setup is totally novel.  In this situation one considers maps from a compact $3$-manifold with conformal triad into a $\Gtwo$-manifold.

\end{section}
\begin{section}{Case: $\Gtwo$-manifolds with the associative triad}\label{SubSecGtwo}
As pointed out in section \ref{SubSecClass}, a $\Gtwo$-manifold possesses the parallel associative triad.  First we describe more fully the subject of $\Gtwo$-manifolds and collect some relevant facts.  An excellent source is \cite{Joyce07}. 
\begin{subsection}{$\Gtwo$-manifolds}

\begin{Defn}
A \altemph{$\Gtwo$-structure} on an oriented $7$-manifold $M$ is a principal, $\Gtwo$-subbundle of the oriented frame bundle of $M$.  
\end{Defn}
One can regard the frame bundle over some point $x \in M$ as consisting of isomorphisms $\phi\colon T_xM \rightarrow \R^7$.  Following \cite{DonaldsonSegal} and \cite{Joyce07}, we describe an equivalent notion of a $\Gtwo$-structure more explicitly.  Consider $V$ a $7$-dimensional real vectorspace with orientation $O$.  There is an open $GL_+(V)$-orbit $P_3 \subset \Lambda^3 V^*$ each element of which has stabilizer $\Gtwo \subset SO(V)$ and likewise for $P_4 \subset \Lambda^4 V^*$.  
Now consider an oriented $7$-manifold $M$.  At each point $p \in M$ we have open subsets $P_{3,p} \subset \Lambda^3T_p^*M, P_{4,p} \subset \Lambda^4T_p^*M$ as before.  Then,
\begin{Defn}
 A \altemph{$\Gtwo$-structure} on an oriented $7$-manifold $M$ is a choice of a $3$-form $\omega$ which lies in $P_{3,p}$ for each $p$.  
\end{Defn}
Note that such a structure defines a $\Gtwo$-structure in the first sense by considering the subbundle of the positive frame bundle consisting of isomorphisms $\alpha\colon T_xM \rightarrow \R^7$ for which $\alpha^*\phi_0 = \phi$.  Conversely, one uniquely defines a $3$-form, Riemannian structure, and Hodge stars since one can define a metric on a $\Gtwo$-manifold by pulling back the Euclidean metric, $\phi_0$ and the Hodge stars, and noting that these are $\Gtwo$-invariant/equivariant.

If $\omega$ is a $3$-form which yields a $\Gtwo$-structure, then one can consider the Riemannian connection $\nabla$ associated to the metric $g$ determined by $\omega$.
\begin{Defn}
 If $\omega$ is a $\Gtwo$-structure on $M$, then the \altemph{torsion} of this $\Gtwo$-structure is $\nabla \omega$.
\end{Defn}
\begin{Prop}(\cite{Joyce07}{11.1.3})\\
Let $\omega$ determine a $\Gtwo$-structure on an oriented $7$-manifold $M$,  and $g$ be the associated metric.  Then the following are equivalent:
 \begin{align*}
 i)& \,\,\, \nabla \omega = 0 \\
 ii)& \,\,\, Hol(g) \subset \Gtwo, \text{  and $\omega$ is the induced $3$-form   } \\
 iii)& \,\,\, d\omega = d^* \omega = 0 \\
 iv)& \,\,\, \text{There exists a parallel spinor field on $M$.} 
\end{align*}
\end{Prop}
The last statement would be ambiguous unless $M$ inherits a particular spin structure from the $\Gtwo$-structure ---this is in fact the case.  $\Gtwo$ embeds in $SO_7$, which induces an embedding of their universal covers.  Because $\Gtwo$ is a simply-connected, embedding is a map $\tilde{\iota} \colon \Gtwo \rightarrow Spin_7$, which is an injective Lie group homomorphism lifting the covering of $SO_7$.  The $\Gtwo$-structure $Q$ (understood as a principal subbundle of the positive frame bundle) induces an $SO_7$ structure on $M$ via $P = SO_7 \cdot Q$.  The spin structure on $M$ is then given by $\tilde{P} = Q \times_{\Gtwo} Spin_7$ making use of $\tilde{\iota}$ for the $\Gtwo$-action on $Spin_7$.  


There are also strong constraints on the topology of closed $\Gtwo$-manifolds.  As modules over $\Gtwo$, we have splittings
\begin{align*}
 &i) &&\Lambda^1T^*M = \Lambda_7^1 & ii)&& \Lambda^2T^*M = \Lambda_7^2 \oplus \Lambda^2_{14} \\
 &iii)&&\Lambda^3T^*M = \Lambda_1^3 \oplus \Lambda_7^3 \oplus \Lambda_{27}^3 & iv)&& \Lambda^4T^*M = \Lambda_1^4 \oplus \Lambda_7^4 \oplus \Lambda^4_{27}\\
&v)& &\Lambda^5 T^*M = \Lambda_7^5\oplus \Lambda_{14}^5 & vi)&& \Lambda^6T^*M = \Lambda_7^6, 
\end{align*}
denoting by the lower indices the dimensions of these submodules.  The Hodge star gives an isometry between $\Lambda_l^k$ and $\Lambda_l^{7-k}$.  $\Lambda_1^3 = \langle  \omega \rangle $, $\Lambda_1^4 = \langle  \star \omega\rangle $, and $\Lambda_7^k$ for all k are canonically isomorphic.  It is a well-known fact that if $M$ has a torsion-free $\Gtwo$-structure, these decompositions are respected by the Hodge Laplacian, yielding a refinement of the de Rham cohomology, $\oplus_{l} H_l^k(M,\R) = H^k(M,\R)$.  We have the theorem
\begin{Thm}
 Let $(M,\omega,g)$ be a compact $\Gtwo$-manifold with torsion-free $\Gtwo$-structure.  Then,
\begin{align*}
 &i)& &H_1^3(M,\R) = \langle [\omega]\rangle , \,\,\,\,\, H_1^4(M,\R) = \langle [\star \omega]\rangle &&\\
 &ii)& &H_l^k(M,\R) \cong  H_l^{n-k}(M,\R),&&\\
 &iii)& &\text{and, if $Hol(g)=\Gtwo$, then } H_7^k(M,\R) = 0 \text{ for all k}&&
\end{align*}
\end{Thm}
Ultimately this theorem means that for a full-holonomy $\Gtwo$-manifold there are only two Betti numbers to be determined ($b^2$,$b^3$).  One can find tables of Betti numbers of known examples in \cite{Joyce07}.  

And finally one might consider the characteristic classes relevant to a $\Gtwo$-manifold.  The next theorem is largely an application of Chern-Weil theory combined with the previous considerations.
\begin{Thm}(\cite{Joyce07}{11.2.7})
 Suppose $M$ is a compact $7$-manifold admitting metrics with full holonomy $\Gtwo$.  Then $M$ is orientable, spin, has finite fundamental group, and has a non-trivial first Pontrjagin class.
\end{Thm}
\end{subsection}
\begin{subsection}{Calibrated geometry and $\Gtwo$-manifolds}
$\Gtwo$-geometry is closely related to the topic of calibrated geometry by virtue of the fact that if a manifold $M$ has a closed $\Gtwo$-structure, i.e. $d\omega=0$, then $\omega$ is a calibration on M.  If $M$ is a $\Gtwo$-manifold (has holonomy inside $\Gtwo$) then the theorem above implies not only that $\omega$ is closed, but that $\star \omega$ is as well.  It follows then that $\star \omega$ is a calibration.  The relevant \emph{calibrated submanifolds} are the integral submanifolds for the distributions consisting of the unit-multivectors dual to these forms.  In what follows we describe this situation explicitly.  Most of this section can be found in \cite{HarveyLawson82}.

We have already noted that a $\Gtwo$-manifold is one whose tangent spaces are identified continuously and orientedly with $\Im\Oct$.  Given the complicated subgroup/subalgebra structure of $\Oct$, it is not surprising that there will be interesting sub-geometries.  In particular, we note that the usual presentation of $\Oct$ is by virtue of the Cayley-Dickson process: 
\begin{equation*}
 \Oct = \mathbb{H} \oplus \mathbb{H} = <1,i,j,k> \oplus <l,il=:I,jl=:J,kl=:K>, \,\,\,\, (a,b)(c,d) := (ac -\bar{d}b, da +b\bar{c}).
\end{equation*}
This product restricts nicely on $\Im \Oct$ to $J(x,y) := \Im(x \cdot y)$.  But it is not associative, hence we have the associator:
\begin{equation*}
 [x,y,z] := (x \cdot y) \cdot z - x \cdot (y \cdot z),
\end{equation*}
which is defined on $\Im\Oct$ and takes values in $\Im\Oct$.  On $\Im\Oct$ one can relate the commutator $[,]$ to $J(,)$ via
\begin{equation*}
 J(x,y) = \frac{1}{2}[x,y].
\end{equation*}
However, the lack of associativity of $\Oct$ means that this commutator doesn't satisfy the Jacobi identity.  In fact,
\begin{equation*}
 J(x,J(y,z)) + J(z,J(x,y)) + J(y,J(z,x)) = -\frac{3}{2}[x,y,z].
\end{equation*}
There are an important class of $3$-planes in $\Im\Oct$ which are the $3$-planes isomorphic to $\Im\mathbb{H}$ inside $\Oct$.  These are the \emph{associative} $3$-planes, the paradigm case being $i \wedge j \wedge k$.  Likewise, a $4$-plane which is complementary to an associative $3$-plane is called \emph{coassociative}, with $l \wedge I \wedge J \wedge K$ an exemplar.  Hence according to the above formulas, the associator vanishes on associative $3$-planes, and on a $\Gtwo$-manifold the associator form is proportional to the \emph{Jacobiator} of $J$ (so that $J$ becomes a Lie bracket on associative $3$-planes).  Let
\begin{equation*}
 G := \{\zeta = x\wedge y \wedge z \text{ an oriented $3$-plane in }\Im\Oct \,\, \vert [x,y,z] = 0  \}
\end{equation*}
Likewise, over a $\Gtwo$-manifold $M$,  we could define the \emph{associative Grassmannian} $G(M)$ which is a fiber bundle over $M$ with fiber over $m \in M$:
\begin{equation*}
 G_x(M) = \{ \zeta = x \wedge y \wedge z \text{ an oriented $3$-plane in }T_mM \,\, \vert [x,y,z]_m = 0  \}
\end{equation*}
We define an \emph{associative submanifold} to be a $3$-dimensional submanifold $X$ of $M$,  such that its tangent 3-planes are in the associative Grassmannian.  These calibrated submanifolds are highly constrained (for instance they are volume-minimizing within their homotopy classes), and yet are locally relatively abundant (for example any analytic surface in a $\Gtwo$-manifold locally extends uniquely to an associative submanifold) \cite{HarveyLawson82}.  Existence theory for compact calibrated submanifolds thus is ultimately a question of whether such open submanifolds ``close up'' smoothly.  
\end{subsection}

\begin{subsection}{Existence}
As one can readily see, in general the MCR equations are an overdetermined system.  Indeed, for maps from $X^3$ to $M^7$ the MCRE is equivalent to a system of eight equations on seven unknowns.  Locally the map is subject to the vector-constraint that any two derivative vectors can be crossed to get the third (with cyclic signs) and the additional constraint that any one of the derivative vectors has the same length.  This fact is the primary issue in the $\R^{n} \rightarrow \R^n$ case of quasiregular maps (hence the Liouville theorem), and it becomes relevant in other cases as well.  

It is exceedingly clear, however, that solutions do likely exist in abundance in particular $\Gtwo$-geometries of interest (likewise for $\Speven$).  There is often a rich world of associative submanifolds of a $\Gtwo$-manifold and several researchers have studied these intensely.  To name one example, Lotay's recent work \cite{Lotay2010} begins a more systematic study of the associative submanifolds of $S^7$ (a \emph{nearly parallel} $\Gtwo$-manifold).  And there is literature pertaining to the existence of $\Gtwo$-manifolds with conic singularities as well as corresponding associatives with singularities.  Given an associative submanifold, the inclusion is multiholomorphic.  One can always obtain another by precomposing with a multiholomorphic endomorphism of the domain or post-composing with an isometry of the target.  The question remains as to whether or not all multiholomorphic maps into a $\Gtwo$-manifold are of this form.  

More speculatively, there may be some special geometric conditions on a $\Gtwo$-manifold which imply that $p$-harmonic maps are either multi-holomorphic or anti-multiholomorphic --analogously to some of the cases in which all harmonic maps from Riemann surfaces to a K\"ahler target are (anti-)holomorphic.  In this setting the $p$-harmonic flow could potentially deform any continuous map homotopically to a $p$-harmonic representative, a la Wei's theorem, \ref{Existence}.  Relaxing the curvature conditions in his theorem would require the analysis of bubbling in limits of the $p$-harmonic flow. 

In the last section we present a rigidity theorem ---a Liouville-type theorem--- for multiholomorphic maps in $\Gtwo$-manifolds.

\end{subsection}
\begin{subsection}{The $\Gtwo$-Liouville Theorem}
In this section we prove a Liouville-type theorem for multiholomorphic maps in the $\Gtwo$-scenario from $X^3$ to $M^7$.  We use the term ``classical'' to refer to the regime in which the regularity of these maps are assumed to be at least $C^3$.  

\begin{Thm}{(The ``Classical,'' $\Gtwo$-Liouville Theorem)}\\
Suppose that $u\colon (X^3,h) \rightarrow (M^7,g)$ is a $C^3$-multiholomorphic map.  On the regular locus of $u$, denote by $\tilde{S}$ the scalar curvature of $u^*g$, by $S$ the scalar curvature of $h$, and by $\lambda_1$ the first eigenvalue of the $h$-Laplacian.  Suppose we have the uniform bounds
\begin{equation*}
 S > -8\lambda_1, \,\,\,\, \tilde{S} \leq 0.
\end{equation*}
Then if $u$ has a critical point, it is a constant map. 
\end{Thm}

\begin{proof}
We start by defining $F(x,u) := \frac{\sqrt{3}}{\lvert du \rvert}u^*\omega([X])$, the function on $X$ which at each point evaluates $u^*\omega$ on a unit volume element in $TX$, scaled by the norm of $du$.  Notice that the MCR equation on $u$ implies,
\begin{equation*}
 u^*g(\bullet,\bullet) = F(x,u)h(\bullet,\bullet).
\end{equation*}
And using the MCR equation, if we apply both sides of this equation to a unit vector in $TX$, we get 
\begin{equation*}
 \frac{1}{\sqrt{3}}\lvert du \rvert^3 = 3 u^*\omega([X]).
\end{equation*}
Hence,
\begin{equation*}
 F = \frac{1}{3}\lvert du \rvert^2.
\end{equation*}

Note that $u$ is orientation-non-reversing which makes it clear that $F$ is in $C^2$.  On the open subset of $\Omega$ on which $F$ is strictly positive, $u^*g$ is conformally equivalent to $h$ by virtue of the conformal factor
\begin{equation*}
 u^*g = e^{2\lambda}h, \,\,\,\,\, \lambda = \frac{1}{2} \log F.
\end{equation*}
Conformal equivalence of metrics imposes constraints on their curvature tensors; specifically we make use of the relation on the Ricci tensors, $\tilde{R}$,and $R$ respectively.  (cf. \cite{Besse}, 1.159)  Let $\nabla$ be the $h$-covariant derivative, and $\Delta$ the $h$-Laplacian $tr_h(\nabla d)$.  Then,
\begin{equation*}
  \tilde{R} = R -(\nabla d\lambda - d\lambda \otimes d\lambda)+ (-\Delta \lambda - \lvert d\lambda \rvert^2)h.
 \end{equation*}
The corresponding scalar curvatures will satisfy
\begin{equation*}
 \tilde{S} = e^{-2\lambda}(S-4\Delta \lambda -2\lvert d\lambda \rvert^2).
\end{equation*}
Repackaging with $\Psi := F^{\frac{1}{4}}$, we can calculate (cf. \cite{Besse}, 1.160),
\begin{equation*}
 \tilde{S}\Psi^5 = -8 \Delta \Psi + S\Psi.
\end{equation*}
This equation applies on each connected component $\Omega^+$ of the positive locus of $\Psi$ inside $\Omega$. 

By hypothesis we have imposed the inequality
\begin{equation*}
  \Delta \Psi -\frac{1}{8}S\Psi = -\frac{1}{8}\tilde{S} \Psi^5 \geq 0.
\end{equation*}
The condition $S > -8 \lambda_1$ implies a solution of 
\begin{equation*}
 \Delta \Psi - \frac{1}{8}S \Psi \geq 0
\end{equation*}
satisfies the classical weak maximum principle in $\Omega^+$ (\cite{GilbargTrudinger}, Cor. 3.2).  Hence, $\Psi$ cannot have a zero on the boundary of $\Omega^+$ unless it is constant.  Thus, $\Omega^+$ is all of $\Omega$ or $\Psi$ is constant.
\end{proof}

In order to deal with less strict regularity/integrability assumptions one must delve more deeply in the theory of functions satisfying the semilinear PDE:
\begin{equation}
 \Delta \Psi = \frac{1}{8}S(x) \Psi- \frac{1}{8}\tilde{S}(x)\Psi^5.
\end{equation}
We note also that the $C^3$-assumption in the previous theorem is important for this argument but not necessarily sharp.
  
\end{subsection}


\end{section}
\vspace{1cm}

\emph{
The author extends his warm gratitude to Jonathan Block who was a perpetually helpful sounding board, and to Denis Auroux, Joachim Krieger, and Clifford Taubes, Katrin Wehrheim, Spiro Karigiannis, and Benoit Charbonneau who were willing to listen to this spiel at least once and lend useful comments.
}
\def\cprime{$'$}

\noindent
University of Waterloo\\ 
Pure Mathematics\\
200 University Avenue West\\
Waterloo, Ontario, N2L 3G1\\ 
Canada\\
aaron.smith@uwaterloo.ca\\
aasmith@alumni.upenn.edu\\

\end{document}